\documentclass{amsart}                 

\usepackage{amssymb,amsmath,amsthm}
\usepackage{hyperref}
\usepackage{color}
\usepackage{graphics}
\usepackage{graphpap}

\theoremstyle{plain}
\newtheorem{theorem}{Theorem}[section]
\newtheorem*{theorem*}{Theorem}
\newtheorem{lemma}[theorem]{Lemma}
\newtheorem{corollary}[theorem]{Corollary}
\newtheorem{proposition}[theorem]{Proposition}

\theoremstyle{remark}
\newtheorem{remark}[theorem]{Remark}
\newtheorem*{remark*}{Remark}

\theoremstyle{definition}
\newtheorem{definition}[theorem]{Definition}
\newtheorem*{definition*}{Definition}

\newtheorem*{notation*}{Notation}

\newtheorem{Basic assumptions}[theorem]{Basic assumptions}

\numberwithin{equation}{section}

\newcount\quantno
\everydisplay{\quantno=0}\everycr{\quantno=0}
\newcommand\quant{\advance\quantno by1
                      \ifnum\quantno=1\qquad\else\quad\fi\forall }
\newcommand\itemno[1]{(\romannumeral #1)}

\renewcommand\Im{\operatorname{\mathrm{Im}}}

\newcommand\Dom{\mathrm{Dom}}

\newcommand\rest[1]{\kern-.1em
          \lower.5ex\hbox{$\scriptstyle #1$}\kern.05em}

\newcommand\set[1]{{\left\{#1\right\}}}

\renewcommand\mod[1]{\vert{#1}\vert}

\newcommand\Bigmod[1]{\Bigl\vert{#1}\Bigr|}

\newcommand\norm[2]{{\Vert{#1}\Vert_{#2}}}

\newcommand\bignorm[2]{\left.{\bigl\Vert{#1}\bigr\Vert_{#2}}\right.}

\newcommand\opnorm[2]{|\!|\!| {#1} |\!|\!|_{#2}}
\newcommand\bigopnorm[2]{\bigl|\!\bigl|\!\bigl| {#1} 
\bigr|\!\bigr|\!\bigr|_{#2}}

\newcommand\prodo[2]{\left\langle#1,#2\right\rangle}

\newcommand\smallfrac[2]{\mbox{\small$\displaystyle\frac{#1}{#2}$}}
\newcommand\wrt{\,\text{\rm d}}

\newcommand\bS{\mathbf{S}}

\newcommand\BC{\mathbb{C}}

\newcommand\BR{\mathbb{R}}

\newcommand\BZ{\mathbb{Z}}

\newcommand\cB{\mathcal{B}}   
\newcommand\frb{\mathfrak{b}}

\newcommand\cD{\mathcal{D}}

\newcommand\cG{\mathcal{G}}   
 
\newcommand\cH{\mathcal{H}}   
\newcommand\frh{\mathfrak{h}} 
\newcommand\cI{\mathcal{I}}   
 
\newcommand\cJ{\mathcal{J}}
    
\newcommand\cL{\mathcal{L}}    
   \newcommand\frm{\mathfrak{m}}  
    
\newcommand\cO{\mathcal{O}}   \newcommand\fro{\mathfrak{o}}

\newcommand\cR{\mathcal{R}}    
\newcommand\cS{\mathcal{S}}  \newcommand\fS{\mathfrak{S}}  
\newcommand\cT{\mathcal{T}}    
\newcommand\cU{\mathcal{U}}    
\newcommand\cV{\mathcal{V}}    
\newcommand\cW{\mathcal{W}}

\newcommand\al{\alpha}
\newcommand\be{\beta}
\newcommand\ga{\gamma}    \newcommand\Ga{\Gamma}
\newcommand\de{\delta}
  \newcommand\vep{\varepsilon}

\newcommand\la{\lambda}   
\newcommand\om{\omega}    \newcommand\Om{\Omega}  
\newcommand\si{\sigma}
\newcommand\te{\theta}

\newcommand\OV{\overline}
\newcommand\funnyk{k\hbox to 0pt{\hss\phantom{g}}}

\newcommand\lu[1]{L^1(#1)}
\newcommand\lp[1]{L^p(#1)}
\newcommand\laq[1]{L^q(#1)}
\newcommand\ld[1]{L^2(#1)}

\newcommand\ly[1]{L^\infty(#1)}

\newcommand\hu[1]{H^1(#1)}

\newcommand\Xu[1]{X^1(#1)}

\newcommand\Xh[1]{X^k(#1)}

\newcommand\Xhs[1]{X_{\si}^k(#1)}
\newcommand\Yu[1]{Y^1(#1)}
\newcommand\Yh[1]{Y^k(#1)}

\newcommand\Yhs[1]{Y_{\si}^k(#1)}

\newcommand\wh{\widehat}
\newcommand\wt{\widetilde}
\newcommand\whH{\widehat{\phantom{G}}\hbox to 0pt{\hss $H$}}

\newcommand\emspace{\hbox to 6pt{\hss}}

\newcommand\rmi{\hbox{\rm (i)}}
\newcommand\rmii{\hbox{\rm (ii)}}
\newcommand\rmiii{\hbox{\rm (iii)}}
\newcommand\rmiv{\hbox{\rm (iv)}}

\newcommand\ir{\int_{-\infty}^{\infty}}

\newcommand\One{{\mathbf{1}}}

\newcommand\e{\mathrm{e}}

\newcommand\Horm{\mathrm{Mih}}

\newcommand\PP{\rm{(I)}}

\newcommand\dest{\text{\rm d}}

\newcommand\Ric{\mathop{\rm Ric}}

\newcommand\Jeta{\cV_{\eta}}

\newcommand\Jetah{\cV_{\eta}^k}
\newcommand\Jetamenoh{\cV_{\eta}^{-k}}

\newcommand\Jbeh{\cU_{\be^2}^k}
\newcommand\Jbemenoh{\cU_{\be^2}^{-k}}

\begin{document}

\title[Hardy spaces on noncompact manifolds]
{Hardy type spaces \\
on certain noncompact manifolds \\
and applications}

\subjclass[2000]{} 

\thanks{This is version 2 of 	arXiv:0812.4209v1. The first version has been revised and rearranged,Ê
with additions, in two papers, of which this new version is the first.
The second paper is posted as arXiv:1002.1161v1}


\author[G. Mauceri, S. Meda and M. Vallarino]
{Giancarlo Mauceri, Stefano Meda and Maria Vallarino}

\address{Giancarlo Mauceri: Dipartimento di Matematica\\ 
Universit\`a di Genova\\
via Dodecaneso 35\\ 16146 Genova\\ Italy 
-- mauceri@dima.unige.it}

\address{Stefano Meda: 
Dipartimento di Matematica e Applicazioni
\\ Universit\`a di Milano-Bicocca\\
via R.~Cozzi 53\\ I-20125 Milano\\ Italy
-- stefano.meda@unimib.it}

\address{Maria Vallarino:
Dipartimento di Matematica e Applicazioni
\\ Universit\`a di Milano-Bicocca\\
via R.~Cozzi 53\\ I-20125 Milano\\ Italy
--  maria.vallarino@unimib.it}

\begin{abstract}
In this paper we consider a complete connected noncompact 
Riemannian manifold $M$ with Ricci curvature bounded from below,
positive injectivity radius and spectral gap $b$.  
We introduce a sequence $\Xu{M}, X^2(M),\ldots$
of new Hardy spaces on $M$, the sequence $\Yu{M}, Y^2(M), \ldots$ 
of their dual spaces, and show that these spaces may 
be used to obtain endpoint estimates for purely imaginary powers
of the Laplace--Beltrami operator and for more general spectral
multipliers associated to the Laplace--Beltrami operator $\cL$ on~$M$.
Under the additional condition that the volume of the geodesic balls of radius
$r$ is controlled by $C\, r^\al \, \e^{2\sqrt{b} r}$ for some real number
$\al$ and for all large~$r$, we prove also
an endpoint result for first order Riesz
transforms $\nabla \cL^{-1/2}$.

In particular, these results apply to Riemannian 
symmetric spaces of the noncompact type. 
\end{abstract}

\maketitle

\setcounter{section}{0}
\section{Introduction} \label{s:Introduction}

The Riesz transform $\nabla (-\Delta)^{-1/2}$
and the purely imaginary powers $(-\Delta)^{iu}$, $u$ in~$\BR$, of 
the Laplacian $\Delta$
are prototypes of singular integral operators on $\BR^n$.  
They are bounded on $\lp{\BR^n}$ for all $p$ in $(1,\infty)$,
and unbounded on $\lu{\BR^n}$ and on $\ly{\BR^n}$ \cite{St2}.
Classical results (see the seminal papers \cite{Ho,FeS}) state 
that singular integral operators satysfying the so called
H\"ormander integral condition are of weak type $1$
and bounded from the Hardy space $\hu{\BR^n}$ to 
$\lu{\BR^n}$ and from $\ly{\BR^n}$ to $BMO(\BR^n)$.  
These results apply, in particular, to $\nabla (-\Delta)^{-1/2}$
and $(-\Delta)^{iu}$.  One reason to choose $(-\Delta)^{iu}$
as an example of singular integral operators
is that it plays a fundamental role in the functional
calculus for $-\Delta$, for functions of the Laplacian
may, at least formally, be reconstructed from $(-\Delta)^{iu}$
via a subordination formula involving the Mellin transform
(see the fundamental works \cite{St1,Co}). 

Now suppose that $M$ is a Riemannian manifold with Riemannian
measure $\mu$, and denote 
by $-\cL$ and $\nabla$ the associated Laplace--Beltrami operator
and covariant derivative respectively.  
It is natural to speculate whether the analogues
of the aforementioned results hold for the operators 
$\nabla \cL^{-1/2}$ and $\cL^{iu}$.  
The multiplier result for generators of semigroups proved in \cite{St1,Co}
applies to $\cL^{iu}$ and gives the $\lp{M}$
boundedness of these operators for $p$ in $(1,\infty)$.
The $\lp{M}$ boundedness of $\nabla \cL^{-1/2}$ for $p$ in $(1,2)$,
and without additional assumptions on $M$,
seems to be a challenging problem, and it is
the object of a very active line of research (see,
for instance, \cite{CD,ACDH} and the references therein).  

As far as endpoint estimates for 
$\nabla \cL^{-1/2}$ and $\cL^{iu}$ are concerned, 
interesting results have been obtained in the case
where $\mu$ is doubling and $M$ satisfies 
some extra assumptions, such as 
appropriate on-diagonal estimate for the heat kernel \cite{CD},
or scaled Poincar\'e inequality~\cite{Ru,MRu,AMR}.
Note that when $\mu$ is doubling, $M$ is a space of homogeneous
type in the sense of Coifman and Weiss, and
a well known theory of atomic Hardy spaces is
available \cite{CW}.

In this paper we consider a complete connected noncompact 
Riemannian manifold $M$ with {Ricci curvature bounded from below},
{positive injectivity radius} and {strictly positive}
bottom $b$ of the spectrum of $\cL$.  It may be worth observing
that under these assumptions the Riemannian measure is nondoubling
and that the volume of geodesic balls in $M$ grow exponentially with the radius.  
Recall that for a Riemannian manifold satisfying 
the above assumptions
there are positive constants $\al$, $\be$ and $C$ such that
\begin{equation} \label{f: volume growth}
\mu\bigl(B(p,r)\bigr)
\leq C \, r^{\al} \, \e^{2\be \, r}
\quant r \in [1,\infty) \quant p \in M,
\end{equation}
where $\mu\bigl(B(p,r)\bigr)$ denotes the Riemannian volume of the 
geodesic ball with centre $p$ and radius~$r$.
Notable examples of such manifolds are nonamenable
connected unimodular Lie groups equipped with a left
invariant Riemannian distance, and symmetric spaces
of the noncompact type with the Killing metric. 

In this setting, weak type $1$ estimates for $\nabla \cL^{-1/2}$ 
and $\cL^{iu}$ are known only when~$M$ is a 
Riemannian symmetric space of the noncompact type \cite{A1,A2,I2,I3,MV}.

Manifolds satisfying the above assumptions
fall into the class of measured metric spaces $X$
considered in \cite{CMM1}, where the authors,
following up earlier works of A.D.~Ionescu~\cite{I1}
and of E.~Russ~\cite{Ru}, 
defined an atomic Hardy space $\hu{X}$ and a space of functions
of bounded mean oscillation $BMO(X)$.
Both $\hu{X}$ and $BMO(X)$ are defined much as in the classical case of
spaces of homogeneous type, 
the only difference being that atoms in the definition of $\hu{X}$
are supported in balls with radius at most $1$,
and that in the definition of $BMO(X)$ 
averages are taken only on balls of radius at most $1$.
As a consequence, they proved that if $\cT$ is bounded on 
$\ld{X}$ and its kernel $k_\cT$ satisfies the following local
H\"ormander's type condition
\begin{equation} \label{f: local HIC}
\sup_{B\in \cB_1} \, \sup_{y\in B}
\int_{(2B)^c} \mod{k_{\cT}(x,y)-k_{\cT}(x,c_B)} \wrt \mu(x)
<\infty,
\end{equation}
where $\cB_1$ denotes the collection of all balls in $X$ of 
radius at most $1$,
then $\cT$ is bounded on $\lp{X}$ for all $p$ in $(1,2]$ 
and from the atomic Hardy space $\hu{X}$ to $\lu{X}$.

The starting point of our work is the 
perhaps surprising fact that when
$\cL$ is the Laplace--Beltrami operator associated to the Killing metric
on Riemannian symmetric spaces of the noncompact type the operators
$\nabla \cL^{-1/2}$ and $\cL^{iu}$,
$u \neq 0$, are unbounded operators from $\hu{M}$ to $\lu{M}$.
The proof of this fact hinges on quite delicate estimates
of the inverse spherical Fourier transform of the associated
multiplier, and will appear in \cite{MMV2}.
Note that, as a consequence, 
their Schwartz kernels $k_{{\cL^{iu}}}$ and $k_{\nabla \cL^{-1/2}}$ do 
not satisfy (\ref{f: local HIC}).

The purpose of this paper is to introduce 
a sequence $\Xu{M}, X^2(M), \ldots$
of new spaces of Hardy type on $M$, 
and the sequence $\Yu{M}, Y^2(M), \ldots$ 
of their dual spaces, and show that these spaces may 
be used to obtain endpoint estimates for $\nabla \cL^{-1/2}$,
$\cL^{iu}$, and for more general spectral multipliers of $\cL$.
The space $\Xh{M}$ is defined as follows.
Denote by $\cU_{\be^2}$ 
the operator $\cL \, (\be^2\cI+ \cL)^{-1}$.  It is straightforward to
check that~$\cU_{\be^2}$ is a bounded injective operator on $\lu{M}+\ld{M}$.  
Denote by $\Xh{M}$ the range of 
the restriction of $\cU_{\be^2}^k$ to $\hu{M}$, endowed with the norm 
$$
\norm{f}{X^k}
= \norm{\cU_{\be^2}^{-k} f}{H^1}.
$$
By definition, each arrow of the following commutative diagram is an
isometric isomorphism of Banach spaces. 

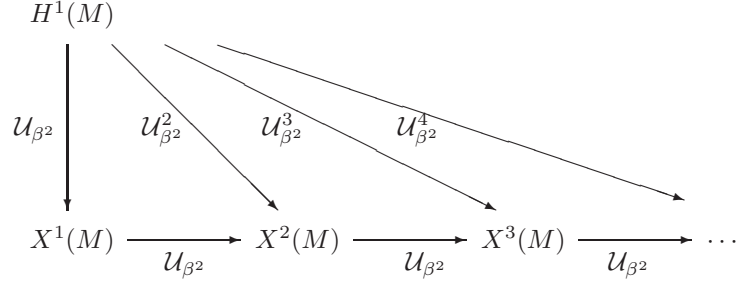
\unitlength1mm
\begin{picture}(100,40)(-55,-20)
\put(-40,15){$\hu{M}$}
\put(-40,-15){$X^1(M)$}
\put(-10,-15){$X^2(M)$}
\put(-35,12){\vector(0,-1){22}}
\put(-27,-14){\vector(1,0){15}}
\put(20,-15){$X^3(M)$}
\put(3,-14){\vector(1,0){15}}
\put(33,-14){\vector(1,0){15}}
\put(50,-15){$\cdots$}
\put(-29,12){\vector(1,-1){22}}
\put(-22,12){\vector(2,-1){44}}
\put(-15,12){\vector(3,-1){62}}
\put(-42,0){$\cU_{\be^2}$}
\put(-25,0){$\cU_{\be^2}^2$}
\put(-9,0){$\cU_{\be^2}^3$}
\put(9,0){$\cU_{\be^2}^4$}
\put(-22,-18){$\cU_{\be^2}$}
\put(10,-18){$\cU_{\be^2}$}
\put(37,-18){$\cU_{\be^2}$}
\end{picture}

\noindent
Thus, $\Xh{M}$ is an isometric copy of $\hu{M}$ 
for each positive integer $k$.
Furthermore,
we shall prove (see Section~\ref{s: Riemannian manifolds}) that
$$
\hu{M} \supset X^1(M) \supset X^2(M)\supset \cdots,
$$
with proper inclusions.  These spaces have nice interpolation properties;
for each positive integer $k$, and for 
every $p$ in $(1,2)$, $\lp{M}$ is an interpolation space 
between $\Xh{M}$ and $\ld{M}$ by the complex method (see Section~\ref{s: New
Hardy}). 
\noindent
The main results of this paper are contained in 
Section~\ref{s: Spectral multipliers on Riemannian manifolds},
and justify, \emph{a posteriori}, the introduction
of the spaces $\Xh{M}$.  In particular,
Theorem~\ref{t: multiplier 2}
states that if $m$ is a holomorphic function in the strip $\bS_\beta=\{ \zeta \in \BC: \Im (\zeta)
\in (-\beta,\beta)\}$ that satisfies 
\begin{equation} \label{f: Horm type estimates I} 
\mod{D^j m(\zeta)} \leq {C} \, 
\max\bigl(\mod{\zeta^2+\be^2}^{-\tau-j},\mod{\zeta}^{-j}\bigr)
\quant \zeta \in \bS_{\be} \quant j \in \{0,1,\ldots,J\},
\end{equation}
for some nonnegative $\tau$ and for a sufficiently large integer $J$,
then $m\bigl(\sqrt{\cL-b}\bigr)$ is bounded from 
$\hu{M}$ to $\lu{M}$ and from $\ly{M}$ to $BMO(M)$
in the case where $b<\be^2$ and from $\Xh{M}$ to
$\hu{M}$ and from $BMO(M)$ to $\Yh{M}$ in the case where $b=\be^2$
and $k> \tau +J$.
This provides, in the case where $b=\be^2$,
endpoint estimates for operators of the form
$\cL^{iu}$ (when $\tau = 0$), but also for ``more singular operators'',
such as $\cL^{iu-\tau} \, (I+ \cL)^{\tau}$, whose
kernels have a comparatively slow decay at infinity.  
We shall call \emph{strongly singular} all the
multipliers satisfying (\ref{f: Horm type estimates I}).
Strongly singular spectral multipliers
were first introduced in \cite{MV}, where 
the authors showed that they satisfy weak type $1$ estimates 
when $M$ is a Riemannian noncompact symmetric spaces.
We remark that the methods of \cite{MV} hinge
on quite precise estimates of the kernel of these
operators, obtained by using the inversion formula
for the spherical Fourier transform.  
Weak type $1$ estimates for such operators seem
out of reach in the more general setting of this paper. 
Note that strongly singular multipliers may have
a rather singular behaviour near the points $\pm i \be$,
and still satisfy an endpoint result for $p=1$.  
We emphasise that this is in sharp constrast with the Euclidean case,
where such a phenomenon cannot occur.

We give applications also to first order Riesz transforms.
It follows from work of T.~Coulhon and X.T.~Duong \cite{CD} 
that, in our setting, the first order Riesz transform $\nabla \cL^{-1/2}$ 
is bounded on $\lp{M}$ for all $p$ in $(1,2]$ and that the translated
Riesz transform $\nabla (\cI+ \cL)^{-1/2}$ is of weak type $1$.  
Russ complemented this result by showing that $\nabla (\cI+ \cL)^{-1/2}$ 
maps $\hu{M}$ into $\lu{M}$.   
Observe that if we consider the part off the diagonal of the kernel of 
$\nabla (\cI+ \cL)^{-1/2}$, then the corresponding  
integral operator is bounded on $\lu{M}$.   
This is no longer true for the kernel of the Riesz transform
$\nabla \cL^{-1/2}$, which
decays much slower at infinity.  Despite this, we prove that if $b = \be^2$,
then $\nabla \cL^{-1/2}$ is bounded from $\Xh{M}$ to $\lu{M}$ for large $k$.
Applications of these spaces to higher order Riesz transforms
associated to the Laplace--Beltrami operator 
on noncompact symmetric spaces
and to multipliers for the spherical Fourier
transform will be considered in a forthcoming paper \cite{MMV2}.  

The space $\Xh{M}$ admits an interesting characterisation in terms of
atoms in $\hu{M}$ that satisfy infinitely many cancellation
conditions.  Its proof, which is rather long,
is deferred to a forthcoming paper \cite{MMV3}.

We now briefly outline the content of the paper. In the next section we define the new Hardy spaces $X^k(M)$ and their duals $Y^k(M)$ in the fairly general framework of the measured metric spaces considered in \cite{CMM1} and show that they have natural interpolation properties. In Section \ref{s: Hardy on manifolds} we specialise to  Riemannian manifolds with Ricci curvature bounded from below, positive injectivity radius and 
strictly positive bottom of the spectrum and we prove some further properties of the new Hardy spaces in this setting. We also state a theorem  on the boundedness on $H^1(M)$ of functions of the Laplacian (Theorem \ref{t: lim huno}), which is of independent interest and plays a crucial role  in the proof of the main results of this paper. The proof of this  theorem is deferred to Section \ref{s: Riemannian manifolds}. The main results of the paper, i.e. the endpoint estimates for strongly singular multipliers and for the Riesz transform are stated and proved in Section \ref{s: Spectral multipliers on Riemannian manifolds}. 

We will use the ``variable constant convention'', and denote by $C,$
possibly with sub- or superscripts, a constant that may vary from place to 
place and may depend on any factor quantified (implicitly or explicitly) 
before its occurrence, but not on factors quantified afterwards. If $\cT$ is a bounded linear operator from the Banach space $A$ to the Banach space $B$, we shall denote by $\bigopnorm{\cT}{A;B}$ its norm. If $A=B$ we shall simply write $\bigopnorm{\cT}{A}$ instead of  $\bigopnorm{\cT}{A;A}$.

\section{New Hardy spaces on metric spaces and interpolation}  
\label{s: New Hardy}

Suppose that $(M,d,\mu)$ is a measured metric space,
and denote by $\cB$ the family of all balls on $M$.
We assume that $\mu(M) > 0$ and that every ball has finite measure.
For each $B$ in $\cB$ we denote by $c_B$ and $r_B$
the centre and the radius of $B$ respectively.  
Furthermore, we denote by $c \, B$ the
ball with centre $c_B$ and radius $c \, r_B$.
For each \emph{scale parameter} $s$ in $\BR^+$, 
we denote by $\cB_s$ the family of all
balls $B$ in $\cB$ such that $r_B \leq s$.  

\begin{Basic assumptions} \label{Ba: NH}
We assume throughout that  
$M$ is \emph{unbounded} and possesses the following properties:
\begin{enumerate}
\item[\itemno1]
\emph{local doubling property} (LD):
for every $s$ in $\BR^+$ there exists a constant $D_s$ 
such that
\begin{equation}  \label{f: LDC} 
\mu \bigl(2 B\bigr)
\leq D_s \, \mu  \bigl(B\bigr)
\quant B \in \cB_s;
\end{equation}
\item[\itemno2]
\emph{isoperimetric property} \hbox{\PP}:
there exist $\kappa_0$ and $C$ in $\BR^+$ such that
for every bounded open set $A$
$$
\mu \Bigl(\bigl\{x\in A: d(x,A^c) \leq \kappa\bigr\}\Bigr)
\geq C \, \kappa\, \mu (A) \quant \kappa\in (0,\kappa_0];
$$
\item[\itemno3] 
\emph{approximate midpoint property} \hbox{(AM)}:
there exist $R_0$ in $[0,\infty)$ and 
$\ga$ in $(1/2,1)$ such that for every pair of points~$x$ and $y$
in $M$ with $d(x,y) > R_0$ there exists a point $z$ in $M$ such that
{$d(x,z) < \ga\, d (x,y)$ and $d(y,z) < \ga\, d (x,y)$; }
\item[\itemno4] 
there is a \emph{semigroup of linear operators} 
$\{\cH^t\}$ acting on $\lu{M}+\ld{M}$ such that 
\begin{itemize}
\item[(a)] the restriction of $\{\cH^t\}$ 
to $\lu{M}$ is a strongly continuous semigroup of contractions;
\item[(b)] the restriction of $\{\cH^t\}$ to 
$\ld{M}$ is strongly continuous, and has \emph{spectral gap} $b>0$, i.e.
$$
\norm{\cH^t f}{2}\le \e^{-b t}\, \norm{f}{2} \quant f\in \ld{M} 
\quant t \in \BR^+;
$$
\item[(c)] $\{\cH^t\}$ is \emph{ultracontractive}, 
i.e. for every $t$ in $\BR^+$ the operator
$\cH^t$ maps $\lu{M}$ into~$\ly{M}$.
\end{itemize}  
\end{enumerate}
\end{Basic assumptions}

\begin{remark}
Assumption~\rmii\ forces $\mu(M) = \infty$.
In fact, it forces $M$ to have exponential volume growth
(see \cite[Proposition~2.5~\rmi]{CMM1} for details).
\end{remark}

\begin{remark} \label{rem: consequences}
Assumption~\rmiv\ has the following straightforward consequences:
\begin{enumerate}
\item[\itemno1]
$\{\cH^t\}$ is a strongly continuous semigroup of 
contractions on $\lu{M}+\ld{M}$;
\item[\itemno2] 
since for each $p$ in $[1,2]$ the space $\lp{M}$ is continuously
embedded in $\lu{M}+\ld{M}$, we may consider the restriction~$\cH_p^t$ 
of the operator $\cH^t$ to $\lp{M}$.
Then $\{\cH_p^t\}$  is strongly continuous on $\lp{M}$, 
and satisfies the estimate
\begin{equation} \label{Lpbounds}
\norm{\cH_p^tf}{p}
\leq \e^{-2b\,(1-1/p)\, t} \, \norm{f}{p} 
\quant f \in \lp{M} \quant t \in \BR^+;
\end{equation} 
\item[\itemno3] 
by \rmiv~(a) and \rmiv~(c) above, 
for each $t$ in $\BR^+$ the operator~$\cH^t$  
maps $\lu{M}$ into $\lu{M}\cap \ld{M}$.  Hence $\cH^t$  
maps $\lu{M}$ into $\lp{M}$ for each $p$ in $[1,2]$.
\end{enumerate}
\end{remark}

Denote by $-\cG$ the infinitesimal generator of $\{\cH^t\}$ on $\lu{M}+\ld{M}$. 
Since $\{\cH^t\}$ is contractive on $\lu{M} + \ld{M}$,
the spectrum of $\cG$ is contained in the right half plane. 
Then, for every $\si$ in $\BR^+$ we may consider the resolvent operator
$(\si\cI+\cG)^{-1}$ of $\{\cH^t\}$, that we denote by $\cR_{\si}$.
We denote by $\cR_{\si,p}$ the restriction of 
$\cR_{\si}$ to $\lp{M}$, and by $-\cG_p$ the generator of $\{\cH^t_p\}$. 
Obviously $\cR_{\si,p}$ is the resolvent of $\{\cH^t_p\}$ and 
$-\cG_p$ is the restriction of $-\cG$ to 
$\Dom(\cG_p)$, which coincides with $\cR_\si\big(\lp{M}\big)$.

For every $\si$ in $\BR^+$ denote by $\cU_\si$ the operator $\cG\cR_\si$. 
Observe that
$$
\cU_\si
= \cI-\si \, \cR_\si,
$$
so that $\cU_{\si}$ is bounded on $\lu{M}+\ld{M}$, and its restriction 
$\cU_{\si,p}$ to $\lp{M}$ is bounded on $\lp{M}$ for every $p\in[1,2]$. 
Moreover $\cU_\si$ and $\cH^t$ commute for every $t$ in $\BR^+$.

\begin{proposition} \label{p: properties of J}
For each positive integer $k$ the following hold:
\begin{enumerate}
\item[\itemno1]
if $p$ is in $(1,2]$, then
the operator $\cU_{\si,p}^k$ is an isomorphism of $\lp{M};$
\item[\itemno2]
the operator $\cU_{\si}^k$ is injective on
$\lu{M} + \ld{M}$.
\end{enumerate}
\end{proposition}

\begin{proof}
First we prove \rmi.  Clearly, it suffices to show that $\cU_{\si,p}$
is an isomorphism of $\lp{M}$.
By (\ref{Lpbounds}) the bottom of the spectrum of $\cG_p$ is positive. 
Thus $\cG_p^{-1}$ and $\si\, \cG_p^{-1}+\cI$ 
are bounded.  Since $\cU_{\si,p}^{-1}=\cG_p^{-1}(\si\cI+\cG_p)$
and $\cG_p^{-1}(\si\cI+\cG_p) = \si\, \cG_p^{-1}+\cI$,  
\rmi\ is proved.  

\medskip
Next we prove \rmii.  It suffices to prove the result in the case
where $k=1$, since the general case follows by induction.
Suppose that $f$ is a function in $\lu{M} + \ld{M}$ 
such that $\cU_{\si} f = 0$.
Then 
$\cU_{\si} \bigl(\cH^tf\bigr)
=\cH^t\bigl(\cU_{\si} f\bigr) 
= 0$ 
for all $t$ in $\BR^+$.   By the ultracontractivity of $\cH^t$, and 
the fact that the restriction of $\cH^t$ to $\ld{M}$ is
bounded on $\ld{M}$, 
the function $\cH^t f$ is in $\ld{M}$ for all $t$ in $\BR^+$. 
Thus $ \cU_{\si} \bigl(\cH^tf\bigr)=\cU_{\si,2} \bigl(\cH^tf\bigr)=0$. 
Hence $\cH^t f=0$, because $\cU_{\si,2}$ is an isomorphism. 
Since $\{\cH^t\}$ is strongly continuous on $\lu{M}+\ld{M}$
by Remark~\ref{rem: consequences}~\rmi,
$\cH^t f$ tends to $f$ in $\lu{M}+\ld{M}$ as $t$ tends to $0$,
and \rmii\ follows.
\end{proof}

We recall the definitions of the atomic Hardy space $H^1(M)$ and its
dual space $BMO(M)$ given in \cite{CMM1}.  

\begin{definition}
An $H^1$-\emph{atom} $a$
is a function in $\lu{M}$ supported in a ball $B$
with the following properties:
\begin{enumerate}
\item[\itemno1]
$\int_B a \wrt \mu  = 0$;
\item[\itemno2]
$\norm{a}{2}  \leq \mu (B)^{-1/2}$.
\end{enumerate}
\end{definition}

\begin{definition}
Suppose that $s$ is in $\BR^+$.  
The \emph{Hardy space} $H_s^{1}({M})$ is the 
space of all functions~$g$ in $\lu{M}$
that admit a decomposition of the form
\begin{equation} \label{f: decomposition}
g = \sum_{k=1}^\infty \la_k \, a_k,
\end{equation}
where $a_k$ is a $H^1$-atom \emph{supported in a ball $B$ of $\cB_s$},
and $\sum_{k=1}^\infty \mod{\la_k} < \infty$.
The norm $\norm{g}{H_s^{1}}$
of $g$ is the infimum of $\sum_{k=1}^\infty \mod{\la_k}$
over all decompositions (\ref{f: decomposition})
of $g$.  
\end{definition}

\noindent
The vector space $H_s^{1}(M)$ 
is independent of~$s$ in $\bigl(R_0/(1-\ga), \infty\bigr)$, 
where $R_0$ and $\ga$ are as in Basic assumptions~\ref{Ba: NH}~\rmiii\
(see \cite[Proposition~5.1]{CMM1}).
Furthermore, given $s_1$ and $s_2$ in $\bigl(R_0/(1-\ga), \infty\bigr)$, 
the norms $\norm{\cdot}{H_{s_1}^{1}}$ and 
$\norm{\cdot}{H_{s_2}^{1}}$ are equivalent.  

\begin{notation*}
{We shall denote the space $H_s^{1}(M)$ simply by $\hu{M}$,
and we endow $\hu{M}$ with the norm $H_{s_0}^{1}(M)$, where
$s_0 = \max \bigl(R_0/(1-\ga), 1\bigr)$.  We note explicitly that
if $R_0 = 0$, then $s_0=1$.}
\end{notation*}

\noindent
The Banach dual of $\hu{M}$ 
is isomorphic \cite[Thm~5.1]{CMM1} to the space $BMO(M)$,
which we now define.

\begin{definition}
The space $BMO(M)$ is the
space of all locally integrable functions~$f$ such that $N(f) < \infty$,
where
$$
N(f)
=  \sup_{B\in \cB_{s_0}} \frac{1}{\mu(B)}
\int_B \mod{f-f_B} \wrt\mu,
$$
and $f_B$ denotes the average of $f$ over $B$. 
We endow $BMO(M)$ with the ``norm''
$$
\norm{f}{BMO}
= N(f).
$$
\end{definition}

\begin{remark}
It is straightforward to check that $f$ is in $BMO(M)$
if and only if its sharp maximal function $f^{\sharp}$,
defined by 
$$
f^{\sharp}(x)
= \sup_{B \in \cB_{s_0}(x)} \frac{1}{\mu(B)}
\int_B \mod{f-f_B} \wrt\mu
\quant x \in M,
$$
is in $\ly{M}$.  Here $\cB_{s_0}(x)$ denotes the family
of all balls in $\cB_{s_0}$ that contain the point $x$. 
\end{remark}

In the last part of this section we define the new 
spaces $\Xhs{M}$ of Hardy type and their dual spaces
$\Yhs{M}$, and prove an interpolation result, which
is relevant for later developments. 

\begin{definition} \label{def: Hardy space}
For each positive integer $k$ and for each $\si$ in $\BR^+$
we denote by $\Xhs{M}$ the Banach space of all
$\lu{M}$ functions~$f$ such that $\cU_{\si}^{-k}f$ is in $\hu{M}$, 
endowed with the norm
$$
\norm{f}{X^k}
= \norm{\cU_{\si}^{-k} f}{H^1}.
$$
\end{definition}

\noindent
Note that $\cU_{\si}^{-k}$ is, by definition, 
an isometric isomorphism between $\Xhs{M}$ and $\hu{M}$.
In Section~\ref{s: Hardy on manifolds}, we shall see that $\Xhs{M}$ may be 
characterised as the image of $\hu{M}$ under a wide class
of maps $\cV^k$.

\begin{remark}
Note that the space $\Xhs{M}$ is continuously 
included in $\lu{M}$.

Indeed, suppose that $f$ is in $\Xhs{M}$.  Then
\begin{align*}
\norm{f}{1}
= \bignorm{\cU_{\si}^k \cU_{\si}^{-k} f}{1} 
\leq \bigopnorm{\cU_{\si}^k}{1}\,  \bignorm{\cU_{\si}^{-k} f}{1} 
&\leq \bigopnorm{\cU_{\si}^k}{1}\,  \bignorm{\cU_{\si}^{-k} f}{H^1} \\
&=    \bigopnorm{\cU_{\si}^k}{1}\,  \bignorm{f}{X_\si^k}, 
\end{align*}
as required.  Note that the last inequality is a consequence of the fact
that $\hu{M}$ is continuously included in $\lu{M}$.
\end{remark}

\begin{definition} \label{def: BMO space}
For each positive integer $k$, and for each $\si$ in $\BR^+$
we denote by $\Yhs{M}$ the Banach dual of $\Xhs{M}$.
\end{definition}

\begin{remark}
Since $\cU_{\si}^{-k}$ is
an isometric isomorphism between $\Xhs{M}$ and $\hu{M}$,
its adjoint map~$\bigl(\cU_{\si}^{-k}\bigr)^*$ is an isometric isomorphism 
between $BMO(M)$ and $\Yhs{M}$.
Hence
$$
\bignorm{\bigl(\cU_{\si}^{-k}\bigr)^*f}{Y_\si^k}
= \norm{f}{BMO}.
$$
\end{remark}

\noindent
Given a compatible couple of Banach spaces $X_0$ and $X_1$ we denote
by $(X_0,X_1)_{[\te]}$ its complex interpolation space,
also denoted by~$X_{\te}$.

\begin{proposition} \label{p: interpolation 1}
Suppose that $(X^0,X^1)$ and $(Y^0,Y^1)$ 
are interpolation pairs of Banach spaces.
Suppose further that $\cT$ is a bounded linear map from 
$X^0 + X^1$ to $Y^0 + Y^1$, such that 
the restrictions $\cT:X^0\to Y^0$ and 
 $\cT:X^1\to Y^1$ are isomorphisms.
Then for every $\te$ in $(0,1)$
the restriction $\cT:X_{\te} \to Y_{\te}$ is an isomorphism.
\end{proposition}

\begin{proof} For every $\te$ in $[0,1]$ denote by $\cT_\theta$ 
the restriction of $\cT$ to $X_\te$. 
Define $\cS:Y_0+Y_1\to X_0+X_1$ by setting 
$$
\cS(y_0+y_1)
= \cT_0^{-1}y_0+\cT_1^{-1}y_1.
$$ 
It is straightforward to check that the operator 
$\cS$ is well defined, bounded and linear. 
Moreover $\cS\cT$ is the identity on $X_0+X_1$ and $\cT\cS$ 
is the identity on $Y_0+Y_1$. 
Thus $\cS=\cT^{-1}$. Hence $\cS_\te=\cT_\te^{-1}$. 
Finally, $\cS_\te:Y_\te\to X_\te$ is bounded by interpolation. 
This concludes the proof of the proposition.
\end{proof}

\begin{theorem} \label{t: interpolation 2}
Suppose that $\si$ is in $\BR^+$, $k$ is a positive integer, and
$\te$ is in $(0,1)$.  The following hold:
\begin{enumerate}
\item[\itemno1]
if $1/p = 1- \te/2$, then $\bigl(\Xhs M , \ld M\bigr)_{[\te]}
= \lp{M}$ with equivalent norms; 
\item[\itemno2]
if $1/q = (1-\te) /2$, then 
$\bigl( \ld M ,\Yhs M\bigr)_{[\te]} = \laq M$ with equivalent norms.
\end{enumerate}
\end{theorem}

\begin{proof}
To prove \rmi, we first observe that $\cU_{\si}^k$ is an isomorphism of
$\hu{M} + \ld{M}$ onto $\Xhs{M} + \ld{M}$.
Then we may apply Proposition~\ref{p: interpolation 1} 
with $\cU_{\si}^k$ in place
of $\cT$, $X^0 = \hu{M}$, $Y^0 = \Xhs{M}$, $X^1 = \ld{M} = Y^1$.  By
\cite[Thm~7.4]{CMM1}
$$
\bigl(\hu M , \ld M\bigr)_{[\te]} = \lp M.
$$
By Proposition~\ref{p: interpolation 1}, the restriction 
of $\cU_{\si}^k$ to $\lp{M}$ is an isomorphism between $\lp{M}$ and 
$\bigl(\Xhs M , \ld M\bigr)_{[\te]}$.  But the restriction 
of $\cU_{\si}^k$ to $\lp{M}$ is just $\cU_{\si,p}^k$, which is an
isomorphism of $\lp{M}$ by Proposition~\ref{p: properties of J}.  Hence 
$\bigl(\Xhs M , \ld M\bigr)_{[\te]}$ and $\lp{M}$ are
isomorphic Banach spaces, as required.

Now \rmii\ follows from \rmi\ by the duality theorem.
\end{proof}

\section{New Hardy spaces on manifolds}
\label{s: Hardy on manifolds}

Suppose that $M$ is a connected $n$-dimensional Riemannian manifold
of infinite volume with Riemannian measure $\mu$.

\begin{Basic assumptions} \label{Ba: on M}
We make the following assumptions on $M$:
\begin{enumerate}
\item[\itemno1] $b>0$;
\item[\itemno2]
$\Ric \geq -\kappa^2$ for some positive $\kappa$ 
and the injectivity radius is positive.
\end{enumerate}
\end{Basic assumptions}

\begin{remark} \label{rem: unif ball size cond}
It is well known that manifolds with properties \rmi-\rmii\ above
satisfy the \emph{uniform ball size condition}, i.e., 
$$
\inf\, \bigl\{ \mu\bigl(B(p, r)\bigr): p \in M \bigr\} > 0
\qquad\hbox{and}\qquad
\sup\, \bigl\{ \mu\bigl(B(p, r)\bigr): p \in M \bigr\} < \infty.
$$
See, for instance, \cite{CMP}, where complete references are given.
\end{remark}

\noindent
Note that manifolds satisfying the assumptions above
also satisfy the Basic assumptions~\ref{Ba: NH}.
Indeed, every length metric space satisfies the \emph{approximate midpoint
property} (AM), and, by standard comparison theorems \cite[Thm~3.10]{Ch},
the measure $\mu$ is \emph{locally doubling}.  
Furthermore, it is known \cite[Section~8]{CMM1}
that for manifolds with Ricci curvature bounded
from below the assumption $b>0$ is equivalent to 
the \emph{isoperimetric property} (I).
Finally, the heat semigroup $\{\cH^t\}$ possesses the properties
\rmiv~(a)--(c) of the Basic Assumptions~\ref{Ba: NH} \cite{Gr}.

In this section we complement the theory developed
in Section~\ref{s: New Hardy} by proving
that the spaces $\Xhs{M}$ and~$\Yhs{M}$, in fact, do not depend
on $\si$ as long as $\si > \be^2-b$
(see Theorem~\ref{t: fundamental properties}).
Our main tool for proving this is a 
$\hu{M}$ boundedness result, of independent interest, for functions
of the Laplace--Beltrami operator on $M$ (Theorem~\ref{t: lim huno}),
which will also play an important role in the proof of Theorem~\ref{t: 
multiplier 2}.

Recall that $-\cL$, $b$ and $\be$ denote the Laplace--Beltrami operator
on $M$, the bottom of the $\ld{M}$ spectrum of $\cL$,
and the exponential rate of growth
of the volume of geodesic balls (see (\ref{f: volume growth})) respectively.
By a result of Brooks \cite{Br} $b\leq \be^2$.
Further, denote by $\de$
a nonnegative number such that the following 
ultracontractive estimate \cite[Section~7.5]{Gr} holds
\begin{equation} \label{f: special}
\bigopnorm{\cH^t}{1;2}
\leq C\, \e^{-bt} \, t^{-n/4} \, (1+t)^{n/4-\de/2}  \quant t \in \BR^+.
\end{equation}
First we define an appropriate function space of holomorphic functions
which will be needed in the statement of Theorem~\ref{t: lim huno}.  

\begin{definition}  \label{d: Hormander at infinity}
Suppose that $J$ is a positive integer and that $W$
is in $\BR^+$.
Denote by~$\bS_{W}$ the strip $\{ \zeta \in \BC: \Im (\zeta)
\in (-W,W)\}$ and by $H^\infty (\bS_{W};J)$ the vector space of 
all bounded \emph{even} holomorphic functions~$f$ in $\bS_{W}$ for which
there exists a positive constant $C$ such that
\begin{equation} \label{f: SsigmaJ}
\mod{D^j f(\zeta)} 
\leq {C} \, {(1+\mod{\zeta})^{-j}}
\quant \zeta\in \bS_{W} \quant j \in \{0,1,\ldots,J\}.
\end{equation}
We denote by $\norm{f}{\bS_{W};J}$ the infimum of all constants $C$ for which (\ref{f: SsigmaJ}) holds.
\end{definition}

\begin{notation*}
{For the sake of notational simplicity, we denote by $\cD$ the
operator $\sqrt{\cL-b}$.}
\end{notation*}

\begin{theorem} \label{t: lim huno}
Assume that $\al$ and $\be$ are as in (\ref{f: volume growth}),
and $\de$ as in (\ref{f: special}). 
Denote by $N$ the integer $[\!\![n/2+1]\!\!] +1$.  
Suppose that $J$ is an integer $>\max\, \bigl(N+2+\al/2-\de,
N + 1/2\bigr)$.
Then there exists a constant $C$ such that 
$$
\opnorm{m(\cD)}{H^1}
\leq C \, \norm{m}{\bS_{\be};J}
\quant m\in H^\infty\bigl(\bS_{\be};J\bigr).
$$
\end{theorem}

\noindent
We emphasise 
that the width of the strip in Theorem~\ref{t: lim huno}
is best possible as the case of symmetric spaces of the noncompact 
type shows \cite{CS}. 
Note that if $M$ is a symmetric space
of the noncompact type with rank $r$ and $\cH^t$ denotes the semigroup 
associated to the Killing metric, then $\de$ is equal to 
the sum of $r/2$ and the cardinality of the positive indivisible restricted roots
\cite[Thm~3.2~\rmiii]{CGM}, and $\al = (r-1)/2$.  Thus, in this
case, we need only to assume $J>N+1/2$ in Theorem~\ref{t: lim huno}. 
Our result may be compared with \cite[Corollary~B.3]{T2}, where the
author proved, under much stronger curvature assumptions on $M$,
that if $m$ is in the symbol class $\cS_{\be^2}^0$, then $m(\cD)$
maps the Goldberg type space $\frh^1(M)$ to $\lu{M}$ and
$\ly{M}$ into $\frb\frm\fro(M)$.

The proof of Theorem~\ref{t: lim huno} is fairly technical and will be given
in Section~\ref{s: Riemannian manifolds}.   
An important consequence of Theorem~\ref{t: lim huno} is that, for fixed $k$,
the spaces $\Xhs{M}$ do not depend on the parameter $\si$, as $\si$
varies in $(\be^2-b,\infty)$.

\begin{theorem} \label{t: fundamental properties}
The following hold:
\begin{enumerate}
\item[\itemno1]
if $\si_1$ and $\si_2$ are in $(\be^2-b,\infty)$, 
then $X_{\si_1}^k(M)$ and $X_{\si_2}^k(M)$ agree as vector spaces, 
and their norms are equivalent;
\item[\itemno2]
if $\si$ is in $(\be^2-b,\infty)$, then
$
\hu{M} \supset X_{\si}^1(M) \supset X_{\si}^2(M) \supset \cdots
$
with continuous inclusions;
\item[\itemno3]
the inclusions in \rmii\ are proper.  
\end{enumerate}
\end{theorem}

\begin{proof}
First we prove \rmi.  Consider the operator $\cT_{\si_1,\si_2}$,
defined on $\ld{M}$ by 
$$
\cT_{\si_1,\si_2}
= \cU_{\si_1}^{-1}\, \cU_{\si_2}. 
$$
Since both $\cU_{\si_1}$ and $\cU_{\si_2}$ are isomorphisms on $\ld{M}$,
so are $\cT_{\si_1,\si_2}$ and $\cT_{\si_1,\si_2}^{-1}$.  
Observe that 
the operators
$\cT_{\si_1,\si_2}$ and 
$\cT_{\si_1, \si_2}^{-1}$ are bounded on $\hu{M}$.  
Indeed,
$$
\cT_{\si_1,\si_2}
= (\si_1\, \cI + \cL) \, (\si_2 \, \cI+ \cL)^{-1} 
= (\si_1-\si_2) \, (\si_2 \, \cI+\cL)^{-1} + \cI.
$$
Hence the boundedness of $\cT_{\si_1,\si_2}$
on $\hu{M}$ is equivalent to that of $(\si_2 \, \cI+\cL)^{-1}$.
To prove that $(\si_2 \, \cI+\cL)^{-1}$
is bounded on $\hu{M}$, it suffices to check that the associated
spectral multiplier $\zeta \mapsto (\si + b + \zeta^2)^{-1}$
satisfies the hypotheses of Theorem~\ref{t: lim huno}.  We omit
the details of this calculation. 
A similar argument shows that 
$\cT_{\si_1,\si_2}^{-1}$ is bounded on $\hu{M}$.  

Thus, $\cT_{\si_1,\si_2}$ is an isomorphism of $\hu{M}$.
Since $\cU_{\si_1}\cT_{\si_1,\si_2} \cU_{\si_2}^{-1} = \cI$,
the identity is an isomorphism between $X_{\si_1}^1(M)$ and $X_{\si_2}^1(M)$,
as required to conclude
the proof of \rmi\ in the case where $k=1$.  The proof in the
case where $k\geq 2$ is similar, and is omitted. 

\medskip
Note that \rmi\ is equivalent to the boundedness of
$\cU_\si$ on $\hu{M}$.  Since $\cU_{\si} = \cI-\si\, (\si\cI+\cL)^{-1}$,
it suffices to prove that the resolvent operator $(\si\cI+\cL)^{-1}$
is bounded on $\hu{M}$.  This has already been done in the proof
of \rmi,  and \rmii\ follows.  

\medskip
Finally we prove \rmiii.  Choose a function $\psi$ in $C_c^\infty(M)$
with nonvanishing integral.  Observe that $\cL \psi$ is a multiple
of a $H^1$-atom, hence it is in $\hu{M}$.  

We shall prove that $\cL^{k+1} \psi$ is in $\Xhs{M}
\setminus X_{\si}^{k+1}(M)$.  Indeed, on the one hand
$$
\cU_{\si}^{-k} \bigl(\cL^{k+1} \psi\bigr)
= (\si\cI + \cL)^k \bigl(\cL \psi\bigr),
$$
which again is a multiple of an $H^1$-atom, hence is in $\hu{M}$.
On the other hand
$$
\cU_{\si}^{-(k+1)} \bigl(\cL^{k+1} \psi\bigr)
= (\si\cI + \cL)^{k+1} (\psi),
$$
which may be written as a linear combination of $\psi$
and of terms of the form $\cL^j \psi$ with $j$ in $\{1,\ldots,k+1\}$.
Therefore the integral of 
$\cU_{\si}^{-(k+1)} \bigl(\cL^{k+1} \psi\bigr)$ does not
vanish, hence it is not in $\hu{M}$ and $\cL^{k+1} \psi$
is not in $X_{\si}^{k+1}(M)$, as required. 
\end{proof}

\begin{definition} \label{def: space Xk}
Suppose that $k$ is a positive integer.  The space $X_{\be^2}^k(M)$
will be denoted simply by $\Xh{M}$.  
\end{definition}

\noindent
By Theorem~\ref{t: fundamental properties},  
for any $\si$ in $(\be^2-b,\infty)$ and each positive integer $k$ we have that
$\Xh{M} = \Xhs{M}$ as vector spaces, and their norms
are equivalent. 

\begin{remark}
The space $\Xh{M}$ may be characterised
as the image of $\hu{M}$ under a wider class of maps.  This is done
in \cite[Subection 4.6]{MMV4}. We briefly describe the result. 

For each positive $\vep$ there exists a function $\eta$ in $C_c(\BR)$ such that
the only zeroes of $1-\wh\eta$ in $\OV{\bS}_{\be+\vep}$ are
the points $\pm i \sqrt b$ (here $\wh\eta$ denotes the Fourier transform of $\eta$).
Suppose that $k$ is a positive integer.  
Denote by $\Jeta$ the operator $\cI-\wh\eta(\cD)$.  The following hold:
\begin{enumerate}
\item[\itemno1]
the map $\Jetah$ is injective on $\lu{M}$;
\item[\itemno2]
$\Jetah \hu{M} = \Xh{M}$ as vector spaces, and the norm
on $\Xh{M}$, defined by
$$
\norm{f}{\eta,k}
= \norm{\Jetamenoh f}{H^1} 
\quant f \in \Xh{M}, 
$$
is equivalent to the norm of $\Xh{M}$.
\end{enumerate}

\end{remark}

\section{Main results}
\label{s: Spectral multipliers on Riemannian manifolds}

In this section we state and prove 
boundedness results for strongly singular spectral multipliers
and first order Riesz transform associated to the Laplace--Beltrami
operator on complete connected Riemannian manifolds $M$ satisfying the
Basic assumptions~\ref{Ba: on M}.   

We recall that in Definition~\ref{d: Hormander at infinity}
we introduced the space $H^\infty (\bS_{W};J)$ 
of functions that are holomorphic and bounded, together 
with their derivatives up to the order $J$,
in the strip $\bS_W$, and satisfy a Mihlin-type condition at infinity.
Here, to deal with a wider class of operators, we define a larger
space of functions that may be singular also
at the points $\pm iW$.

\begin{definition}  \label{defi: HinftySbetakappa}
Suppose that $J$ is a positive integer, that 
$\tau$ is in $[0,\infty)$, and that~$W$ is in~$\BR^+$.
The space $H (\bS_{W};J,\tau)$ is the vector space of 
all holomorphic \emph{even} functions $f$ in the strip 
$\bS_{W}$ for which
there exists a positive constant $C$ such that
\begin{equation} \label{f: Ssigmakappa}
\mod{D^j f(\zeta)} \leq {C} \, 
\max\bigl(\mod{\zeta^2+W^2}^{-\tau-j},\mod{\zeta}^{-j}\bigr)
\quant \zeta \in \bS_{W} \quant j \in \{0,1,\ldots,J\}.
\end{equation}
We denote by~$\norm{f}{\bS_{W};J,\tau}$ the infimum of all constants $C$ for which (\ref{f: Ssigmakappa}) holds.
\end{definition}

\noindent
Note that, for each fixed $j$, the right-hand side of 
(\ref{defi: HinftySbetakappa}) is infinite of order $-\tau-j$
at $\pm iW$, and vanishes of order $j$ at infinity.  
Thus, if $\tau=0$, and $f$ is in $H (\bS_{W};J,\tau)$, then 
$f$ satisfies Mihlin-type conditions both near the points $\pm iW$
and at infinity.  In particular, 
the derivatives of $f$ may be unbounded in any 
neighbourhood of $iW$, and of $-iW$. 
Finally, if $\tau$ is in $\BR^+$, and $f$ is in $H (\bS_{W};J,\tau)$,
then both $f$ and its derivatives up to the order $J$ may be unbounded in any 
neighbourhood of $iW$, and of $-iW$. 

\begin{remark}
An interesting example of a function in $H(\bS_{\be};J,\tau)$ is 
$$
m(\zeta)
= (\zeta^2+\beta^2)^{-iu-\tau} \, (\zeta^2+\beta^2+1)^{\tau},
$$
where $\tau$ is in $[0,\infty)$.  Note that if $b= \be^2$, then 
$m(\cD) = \cL^{-iu-\tau} \, (\cL+\cI)^{\tau}$.
It is worth observing that there are no endpoint results
at $p=1$ for this operator in the literature when $\tau >1$.
In the case where $M$ is a symmetric space of the noncompact type, 
it is known \cite{A1,AJ,MV} that $m(\cD)$ is of weak type $1$ if and only if
$\tau \leq 1$, but the proof of this fact uses the spherical
Fourier transform and very specific information on the structure
of the symmetric space, and it is hardly extendable.   
\end{remark}

\begin{theorem} \label{t: multiplier 2}
Assume that $\al$ and $\be$ are as in (\ref{f: volume growth}),
and $\de$ as in (\ref{f: special}).  
Suppose that $\tau$ is in $[0,\infty)$, that 
$J$ and $k$ are integers, with $k> \tau+J$
and $J > \max\, \bigl(N+ 2+\al/2-\de, N + 1/2\bigr)$,
where $N$ denotes the integer $[\!\![n/2+1]\!\!] +1$.
The following hold:
\begin{enumerate}
\item[\itemno1]
if $b < \be^2$, then there exists a constant $C$ such that 
$$
\opnorm{m(\cD)}{H^1;L^1} 
\leq C \,  \norm{m}{\bS_{\be;J,\tau}}
\quant m \in H (\bS_{\be};J,\tau)
$$
and 
$$
\opnorm{m(\cD)^t}{L^\infty;BMO} 
\leq C \,  \norm{m}{\bS_{\be;J,\tau}}
\quant m \in H (\bS_{\be};J,\tau),
$$
where $m(\cD)^t$ denotes the transpose operator of $m(\cD)$;
\item[\itemno2]
if $b = \be^2$, then there exists a constant $C$ such that 
$$
\opnorm{m(\cD)}{X^k;H^1} 
\leq C \,  \norm{m}{\bS_{\be;J,\tau}}
\quant m \in H (\bS_{\be};J,\tau)
$$
and 
$$
\opnorm{m(\cD)^t}{BMO;Y^k} 
\leq C \,  \norm{m}{\bS_{\be;J,\tau}}
\quant m \in H (\bS_{\be};J,\tau),
$$
where $m(\cD)^t$ denotes the transpose operator of $m(\cD)$.
\end{enumerate}
\end{theorem}

\begin{proof}
First we prove \rmi.  Consider the map $\wt\cU$, defined by
$$
\wt\cU
= \bigl[\cL+(\be^2-b)\cI\bigr] \, (\be^2\cI + \cL)^{-1}. 
$$
Observe that $\wt\cU = \cI - b \, (\be^2 \cI+ \cL)^{-1}$ extends
to a bounded operator on $\lu{M}$, because the $\lu{M}$-spectrum
of $\cL$ is contained in the right half-plane.  Similarly,
the operator $\cI + b \, [(\be^2-b)\cI+ \cL]^{-1}$
extends to a bounded operator on $\lu{M}$; it is straightforward
to check that this operator is the inverse of $\wt\cU$ on $\lu{M}$.
Thus, $\wt\cU$ is an isomorphism of $\lu{M}$, and so is $\wt\cU^k$.  

Consequently, $m(\cD)$ is bounded from $\hu{M}$ to $\lu{M}$
if and only if $\wt\cU^k m(\cD)$ is bounded~from 
$\hu{M}$ to $\lu{M}$.  Observe that
$\wt\cU^k m(\cD) = u_k(\cD)$, where 
$$
u_k(\zeta)
= \Bigl(\frac{\zeta^2+\be^2}{\zeta^2+b+\be^2}\Bigr)^k \, m(\zeta).
$$
It is straightforward to check that 
there exists a constant $C$ such that 
$$
\mod{D^j u_k(\zeta)}
\leq C\, \norm{m}{\bS_{\be};J,\tau} \, \bigl(1+\mod{\zeta}\bigr)^{-j}
\quant \zeta \in \bS_{\be}
\quant j \in \{0,1,\ldots,J\}.
$$
Here we use the fact that $k>\tau + J$.
Thus, $u_k(\cD)$ is bounded on $\hu{M}$ by Theorem~\ref{t: lim huno}, 
hence from $\hu{M}$ to $\lu{M}$, as required to prove the first estimate.

The second follows from the first by a duality argument.

\medskip
Next we prove \rmii.  Observe that $m(\cD) = m(\cD)\,\Jbeh\,\Jbemenoh$.
Since $\Jbemenoh$ is an isometric isomorphism between $\Xh{M}$ and
$\hu{M}$, to prove that $m(\cD)$ is bounded from $\Xh{M}$
to $\hu{M}$ it suffices to show that 
the operator $ m(\cD)\,\Jbeh$
extends to a bounded operator on $\hu{M}$.
Note that 
$m(\cD)\,\Jbeh = v_k(\cD)$, where 
$$
v_k(\zeta) 
=  \Bigl(\frac{\zeta^2+b}{\zeta^2+b+\be^2}\Bigr)^k \, m(\zeta).
$$
It is straightforward to check that 
there exists a constant $C$ such that 
$$
\mod{D^j v_k(\zeta)}
\leq C\, \norm{m}{\bS_{\be};J,\tau} \, \bigl(1+\mod{\zeta}\bigr)^{-j}
\quant \zeta \in \bS_{\be}
\quant j \in \{0,1,\ldots,J\}.
$$
Here we use the fact that $k>\tau + J$.
Thus, $v_k(\cD)$ is bounded on $\hu{M}$ by Theorem~\ref{t: lim huno}, as required to prove
the first estimate.
The second follows from the first by a duality argument.

The proof of the theorem is complete. 
\end{proof}

\begin{remark}
Assume that $M$ has $C^\infty$ bounded geometry.
By proceeding as in the proof of Theorem~\ref{t: multiplier 2}
and using \cite[Thm~10.2]{CMM1} instead Theorem~\ref{t: lim huno},
we may prove Theorem~\ref{t: multiplier 2}~\rmi\ with
$J> \max (\al+1, n/2+1)$ in place of 
$J > \max\, \bigl(N+ 2+\al/2-\de, N + 1/2\bigr)$.
\end{remark}

\begin{corollary} \label{c: symmetric spaces}
Suppose that $M$ is a symmetric space of the noncompact type
and that $-\cL$ is the Laplace--Beltrami operator
with respect to the Killing metric.  If $k> n/2 +3$, then
$\cL^{iu}$ is bounded from $\Xh{M}$ to $\hu{M}$.  
\end{corollary}

\begin{proof}
Indeed, it is well known that $\al = (r-1)/2$, where $r$ is the rank of
the symmetric space, and $\de = \upsilon + r/2$, where $\upsilon$
denotes the cardinality of the indivisible positive restricted roots.
Notice that $3/2+\al/2-\de \leq 0$, so that the hypotheses of 
Theorem~\ref{t: multiplier 2} are satisfied whenever $J>n/2+2$ 
and $k>J$, and the required conclusion follows. 
\end{proof}

We conclude this section with the following endpoint result
for the first order Riesz transform.  Our method hinges on the
fact that if $b=\be^2$ and $k$ is large enough, then
the operator $\cL^k\, (\be^2 \cI+ \cL)^{-k}$ 
is bounded on $\hu{M}$ by Theorem~\ref{t: lim huno}.

\begin{theorem}
Assume that $\al$ and $\be$ are as in (\ref{f: volume growth}), 
and $\de$ as in (\ref{f: special}).  
Suppose that $b=\be^2$ and that $k$ is an integer
$>\max\, \bigl(N+2+\al/2-\de, N + 1/2\bigr)$,
where $N$ denotes the integer $[\!\![n/2+1]\!\!] +1$.  
Then the first order Riesz transform 
$\nabla \cL^{-1/2}$ is bounded from $\Xh{M}$ to $\lu{M}$.
\end{theorem}

\begin{proof}
Since $\cL^k\, (\be^2 \cI+ \cL)^{-k}$ is an isometry
between $\hu{M}$ and $\Xh{M}$, it suffices to prove
that $\nabla \cL^{k-1/2}\, (\be^2 \cI+ \cL)^{-k}$
is bounded from $\hu{M}$ to $\lu{M}$.   Observe that
$$
\nabla \cL^{k-1/2}\, (\be^2 \cI+ \cL)^{-k}
= \nabla(\be^2 \cI+ \cL)^{-1/2} \, \cL^{k-1/2}\,(\be^2 \cI+ \cL)^{1/2-k}. 
$$
The right hand side is the composition of the operators
$\cL^{k-1/2}\,(\be^2 \cI+ \cL)^{1/2-k}$, which
is bounded on $\hu{M}$ by Theorem~\ref{t: lim huno}, and 
of the translated Riesz transform
$\nabla(\be^2 \cI+ \cL)^{-1/2}$, which is bounded from 
$\hu{M}$ to $\lu{M}$ by \cite{Ru}.  The required result follows.
\end{proof}

\section{Operators bounded on $\hu{M}$}
\label{s: Riemannian manifolds}

This section is devoted to the proof of Theorem 3.5 and is divided in the following subsections:  
Subsection~\ref{subsec: Some lemmata}, which contains 
few preliminary results in one dimensional Fourier
analysis;
Subsection~\ref{subsec: Mean}, where we explain the r\^ole
of the wave propagator in the decomposition  into atoms
of the image $\cT a$ of an $H^1$-atom $a$ by an operator $\cT$;
Subsection~\ref{subsec: Economical}, where we prove an economical 
decomposition of $H^1$-atoms with ``big'' support into 
$H^1$-atoms with support in balls in $\cB_1$;  
Subsection~\ref{subsec: Hone}, where we prove Theorem~\ref{t: lim huno}.

\subsection{Some lemmata} \label{subsec: Some lemmata}

This subsection contains a few technical lemmata concerning 
one-dimensional Fourier analysis. 
Some related material may be found in 
\cite[Subsection~2.3]{MMV1}, which we shall sometimes refer to,
for a discussion of the motivations behind this rather
technical development.

For every $f$ in $\lu{\BR}$ 
define its Fourier transform $\widehat f$ by
$$
\widehat f(t) = \ir f(s) \, e^{-ist} \wrt s
\quant t \in \BR.
$$
Suppose that $f$ is a function on $\BR$, and that $\la$ is in $\BR^+$.
We denote by $f^\la$ and $f_\la$ the~$\la$-dilates of $f$, defined by
\begin{equation} \label{f: dilate} 
f^\la(x)
= f(\la x)  
\qquad\hbox{and}\qquad
f_\la(x)
= \la^{-1} \, f(x/\la)  \quant x \in \BR.
\end{equation}

\noindent
For each $\nu\geq -1/2$, denote by $\cJ_\nu: \BR\setminus\{0\} \to \BC$
the modified Bessel function of order $\nu$, defined by
$$
\cJ_\nu(t) = \frac{J_\nu(t)}{t^\nu},
$$
where $J_\nu$ denotes the standard Bessel function of the first
kind and order $\nu$ (see, for instance, 
\cite[formula~(5.10.2), p.~114]{L} for the definition).  Recall that 
$$
\cJ_{-1/2} (t)
= \sqrt{\frac{2}{\pi}}\  \cos t
\qquad\hbox{and that}\qquad
\cJ_{1/2} (t)
= \sqrt{\frac{2}{\pi}}\ \frac{\sin t}{t}. 
$$

\noindent
For each positive integer $\ell$,
we denote  by $\cO^\ell$ the differential operator 
$t^\ell\, D^\ell$ on the real~line.

\begin{lemma}\label{l: P}
For every positive integer $k$ there exists a polynomial 
$P_{k+1}$ of degree $k+1$ without constant term, such that 
\begin{equation} \label{f: propertiesRiesz I} 
\ir {f} (t) \, \cos (vt) \wrt t
=   \ir  P_{k+1}(\cO)f(t)\,  \cJ_{k+1/2} (t v)  \wrt t,
\end{equation}
for all functions $f$  such that 
$\cO^\ell f\in \lu{\BR}\cap C_0(\BR)$ for all
$\ell$ in $\{0,1,\ldots, k+1\}$.
\end{lemma} 

\begin{proof}
The proof uses the definition and some properties of the
generalised Riesz means $R_{d,z}$, introduced in \cite[Section~1]{CM}.
We refer the reader to \cite[Section~2]{MMV1} for all
the prerequisites needed here.
In particular, recall that 
$R_{3+2k,0} =  R_{3+2k,-k}R_{3,k}$ by \cite[Lemma~2.3~\rmi]{MMV1}.  
Now, by integrating by parts
and using \cite[Lemma~2.3~\rmi\ and \rmii]{MMV1},
$$
\begin{aligned}
\ir {f} (t) \, \cos (vt) \wrt t
& = - \ir \cO f(t) \, \, \frac{\sin(vt)}{vt} \, \wrt t \\
& = -\sqrt{\frac{\pi}{2}}\, \ir \cO f(t) \, \, \bigl(R_{3+2k,0} 
   \cJ_{1/2}^v\bigr)(t)\, \wrt t \\
& = -\sqrt{\frac{\pi}{2}}\,\ir R_{3+2k,-k}^*\bigl(\cO {f} \bigl)(t)\, 
     \bigl(R_{3,k}\cJ_{1/2}^v\bigr) (t) \wrt t
\end{aligned}
$$
for all $v$ in $\BR$. Furthermore, the definitions of $R_{3,k}$
and of $\cJ_{1/2}$ and an integration by parts show that
$$
\begin{aligned}
\bigl(R_{3,k} \cJ_{1/2}\bigr) (u)
& =  \frac{2}{\Ga(k)} \, \frac{1}{u} \, \int_0^1 s \, (1-s^2)^{k-1} \, 
       \sqrt{\smallfrac{2}{\pi}} \, \sin(su) \wrt s \\
& =  \sqrt{\frac{2}{\pi}}\, \,  \frac{1}{\Ga(k+1)} \, 
       \int_0^1 (1-s^2)^{k} \, \cos(su) \wrt s \\
& =  2^k \, \cJ_{k+1/2}(u).
\end{aligned}
$$
By \cite[Lemma~2.4~\rmi]{MMV1} there exist constants $c_\ell$ such that 
$R_{3+2k,-k}^*\bigl(\cO {f} \bigl) = \sum_{\ell=0}^k c_\ell \, \cO^{\ell+1} f$,
so that
$$
\begin{aligned}
\ir {f} (t) \, \cos (vt) \wrt t
& =  \sum_{\ell=0}^k  c_\ell'
     \ir  \cO^{\ell}\bigl(\cO  {f}\bigr)(t)\,  
      \cJ_{k+1/2} (t v)  \wrt t,
\end{aligned}
$$
and the required formula, with
$P_{k+1} (s) = \sum_{\ell=0}^k  c_\ell' \, s^{\ell+1}$, follows.
\end{proof}

\begin{remark} \label{rem: P^*} 
We shall denote by $P_{k+1}(\cO)^*$ the formal 
adjoint of the operator $P_{k+1}(\cO)$, i.e. the operator defined by
$$
\int_{-\infty}^\infty f(t)\,P_{k+1}(\cO)^* g(t)\wrt t
= \int_{-\infty}^\infty P_{k+1}(\cO) f(t)\, g(t)\wrt t
\quant f, g\in C^\infty_c(\BR).
$$
Note that $P_{k+1}(\cO)^*$ is still a polynomial of degree $k+1$ in 
$\cO$ and that \break $P_{k+1}(\cO)^*\cJ_{k+1/2}(vt)=\cos(vt)$, 
by (\ref{f: propertiesRiesz I}).
\end{remark}

Denote by $\om$ an even function in 
$C_c^\infty(\BR)$ which is supported in $[-3/4,3/4]$, is equal to~1
in $[-1/4,1/4]$, and satisfies 
$$
\sum_{j\in \BZ} \om(t-j) = 1
\quant t \in \BR.
$$
Denote by $\phi$ the function $\om^{1/4}- \om$, 
where $\om^{1/4}$ denotes the $1/4$-dilate of $\om$.  
Then $\phi$ is smooth, even and vanishes in the complement of the set
$\{t \in \BR: 1/4\leq \mod{t} \leq 4\}$.
For a fixed $R$ in $(0,1]$ and for each positive integer $i$, 
denote by $E_{i}$ the set
$\{t \in \BR: 4^{i-1}R \leq \mod{t} \leq 4^{i+1}R\}$.
Clearly $\phi^{1/(4^{i}R)}$ is supported in $E_{i}$,
and $\sum_{i=1}^\infty \phi^{1/(4^{i}R)} =~1$ in $\BR\setminus (-R,R)$.
Denote by $d$ the integer $[\!\![\log_4(3/R)]\!\!]+1$. 
To avoid cumbersome notation, we write $\rho_i$ instead of $1/(4^{i}R)$.
Then
\begin{equation} \label{f: dec om phi}
\om^{\rho_0} +  \sum_{i=1}^{d} \phi^{\rho_i} =~1 
\qquad\hbox{on}\quad [-3,3].
\end{equation}

\begin{definition}
We say that a function $g:\BR\to \BC$ satisfies a
\emph{Mihlin condition}~\cite{Ho} of order $J$ at infinity on the real line if 
there exists a constant $C$ such that 
\begin{equation} \label{f:Hormandercondition}
\mod{D^\ell g(t)} 
\leq C\, (1+\mod{t})^{-\ell}
\quant t\in \BR \quant \ell \in \{0,\ldots,J\}.
\end{equation}
We denote by $\norm{g}{\Horm(J)}$ the infimum of all constants $C$ for which (\ref{f:Hormandercondition})
holds.
\end{definition}

\begin{lemma} \label{l: technical nuovo}
Suppose that $k$ is a nonnegative integer, and that
$K$ is an even tempered distribution on $\BR$ such that 
$\norm{\wh K}{\Horm(k+2)}$ is finite.
The following hold:
\begin{enumerate}
\item[\itemno1]
for each $\ell$ in $\{0,\ldots,k\}$ 
the function $t\, \cO^\ell K$ is in $\ly{\BR}$, 
and there exists a constant $C$ such that
$$
\norm{t\, \cO^\ell K}{\infty}
\leq C \, \norm{\wh K}{\Horm(k+2)}
\quant \ell \in \{0,\ldots,k\};
$$
\item[\itemno2]
if $k\geq 1$ and the support of $K$ is contained
in $[-1,1]$, then
$
\wh K 
= \sum_{i=0}^{d} S_i,
$
where the functions $S_i: \BR\to\BC$ are defined by
\begin{equation} \label{f: A0}
S_0(\la)
=  (\wh\om_{\rho_0}*\wh K)(\la)
  +  \sum_{j=1}^k c_{j,k}\, \ir K(t) \, 
 \cO^j\om(\rho_0t) \,  \cO^{k-j}\cJ_{k+1/2}(\la t) \wrt t
\end{equation}
for suitable constants $c_{j,k}$, and, for $i$ in $\{1,\ldots,d\}$,
\begin{equation}  \label{f: Ai}
S_i(\la)
= \frac{1}{2\pi} \, \ir \phi^{\rho_i}(t)\, 
    P_{k+1}(\cO) K(t) \, \cJ_{k+1/2}(\la t) \wrt t;
\end{equation}
\item[\itemno3]
if the support of $K$ is contained
in $[-1,1]$, then there exists a constant $C$ such that
$$
\norm{S_0}{\infty}
\leq C \, \norm{\wh K}{\Horm(2)}.
$$
\end{enumerate}
\end{lemma}

\begin{proof}
First we prove \rmi\ in the case where $k=0$.  
Since $\wh K$ satisfies a Mihlin condition
of order $2$ at infinity, $D^2 \wh K$ is in $\lu{\BR}$ (see
(\ref{f:Hormandercondition})), and 
we may define $F: \BR\to \BC$ by
$$
F(t)
= \ir D^2 \wh K(\zeta) \, \e^{i\zeta t} \wrt \zeta.
$$
By elementary Fourier analysis
$
t K(t)
= -{t}^{-1} \, F(t).
$
Observe that $F(0) = 0$, because
$$
\begin{aligned}
F(0)
& = \lim_{A \to \infty} \int_{-A}^A D^2 \wh K(\zeta) \wrt \zeta \\
& = 2 \, \lim_{A \to \infty} D \wh K(A) \\
& = 0,
\end{aligned}
$$
where we have used the fact that $K$ is even and $D \wh K$ vanishes at infinity,
because $\norm{\wh K}{\Horm(2)}$ is finite.
Furthermore
$$
\begin{aligned}
F(t)
& = F(t) - F(0) \\
& = \ir D^2 \wh K(\zeta) \, (\e^{i\zeta t} -1) \wrt \zeta.
\end{aligned}
$$
Suppose that $t$ is positive.
Then we write the last integral as the sum of the integrals
over the sets $\{\zeta \in \BR: \mod{\zeta} \leq 1/t\}$
and $\{\zeta \in \BR: \mod{\zeta} > 1/t\}$, and estimate them
separately.

To treat the first we integrate by parts, and obtain
$$
\begin{aligned}
&  \int_{\mod{\zeta}\leq 1/t}
       D^2 \wh K(\zeta) \, (\e^{i\zeta t} -1) \wrt \zeta  \\
&  = D\wh K (1/t) \, (\e^i-1) -  D\wh K (-1/t) \, (\e^{-i}-1) 
      - it \int_{\mod{\zeta}\leq 1/t} D \wh K(\zeta) \, 
        \e^{i\zeta t} \wrt \zeta.
\end{aligned}
$$
Since $D \wh K$ is odd, its integral over $[-1/t,1/t]$ vanishes,
so that the last integral may be rewritten as 
$$
\int_{\mod{\zeta}\leq 1/t} D \wh K(\zeta) \, 
(\e^{i\zeta t}-1) \wrt \zeta.
$$
Hence
$$
\begin{aligned}
& \Bigmod{\int_{\mod{\zeta}\leq 1/t}
       D^2 \wh K(\zeta) \, (\e^{i\zeta t} -1) \wrt \zeta}  \\
&  \leq C\, \norm{\wh K}{\Horm(2)} \,  \frac{\mod{t}}{1+\mod{t}} \, 
      +C \, t^2 \int_{\mod{\zeta}\leq 1/t} 
        \mod{\zeta\, D \wh K(\zeta)}  \wrt \zeta \\
&  \leq C\, \norm{\wh K}{\Horm(2)} \, \mod{t} \quant t \in \BR^+.
\end{aligned}
$$
To estimate the second, write
$$
\begin{aligned}
\Bigmod{\int_{\mod{\zeta}> 1/t}
   D^2 \wh K(\zeta) \, (\e^{i\zeta t} -1) \wrt \zeta}  
& \leq C\, \norm{\wh K}{\Horm(2)}  \, \int_{\mod{\zeta}>1/t} 
   \frac{1}{1+\zeta^2} \wrt \zeta \\
&  \leq C\, \norm{\wh K}{\Horm(2)} \, \mod{t} \quant t \in \BR^+.
\end{aligned}
$$
Finally, since $K$ is even,
$$
\begin{aligned}
\norm{t K}{\infty}
& \leq \sup_{t \in \BR}  \frac{\mod{F(t)}}{\mod{t}} \\
& \leq C \, \norm{\wh K}{\Horm(2)},
\end{aligned}
$$
as required to conclude the proof of \rmi\ in the case where $k=0$. 

Next we assume that $k\geq 1$.  By the case $k=0$ applied to
$\cO^\ell K$, we see that
$$
\norm{t\, \cO^\ell K}{\infty}
\leq C \, \norm{\wh {\cO^\ell K}}{\Horm(2)}.
$$
Since
$
\wh {\cO^\ell K} 
= \sum_{j=0}^\ell \al_{j,\ell} \, \cO^j \wh K
$
for suitable constants $\al_{j,\ell}$,  
$$
\begin{aligned}
\norm{\wh {\cO^\ell K}}{\Horm(2)} 
& \leq C \, \sum_{j=0}^\ell \norm{\cO^j\wh K}{\Horm(2)} \\
& \leq C \, \norm{\wh K}{\Horm(2+\ell)},
\end{aligned}
$$
which is clearly dominated by $C \, \norm{\wh K}{\Horm(k+2)}$, 
as required to conclude the proof of~\rmi.

\medskip
Now we prove \rmii.   
Suppose that $\vep$ is in $(0,1)$.  Clearly $\wh K(\la)$ is the limit
of $(\wh\om^{\vep}\, \wh K)(\la)$ as~$\vep$ tends to $0$.
By Fourier inversion formula and 
Lemma \ref{l: P}
$$
\begin{aligned}
(\wh\om^{\vep}\, \wh K)(\la)
&  = \frac{1}{2\pi} \, \ir \om_\vep *K(t) \, \cos(\la t) \wrt t \\
&  = \frac{1}{2\pi}\, \ir P_{k+1}(\cO)(\om_\vep *K)(t) 
     \, \cJ_{k+1/2}(\la t) \wrt t 
\quant \la \in \BR.
\end{aligned}
$$
We write the right-hand side as $\sum_{i=0}^{d} S_i(\la;\vep)$, where
\begin{equation} \label{f: defin equiv di S0}
S_0(\la;\vep) 
= \frac{1}{2\pi}\, \ir \om^{\rho_0}(t)\, 
    P_{k+1}(\cO)(\om_\vep *K)(t) \, \cJ_{k+1/2}(\la t) \wrt t
\quant \la \in \BR, 
\end{equation}
and, for each $i$ in $\{1,\ldots,d\}$, 
$$
S_i(\la;\vep) 
= \frac{1}{2\pi} \, \ir \phi^{\rho_i}(t)\, 
   P_{k+1}(\cO)(\om_\vep *K)(t) \, \cJ_{k+1/2}(\la t) \wrt t
\quant \la \in \BR. 
$$
Observe that
$$
S_0(\la;\vep) 
= \frac{1}{2\pi} \, \ir 
  (\om_\vep *K)(t) \, \, P_{k+1}(\cO)^*(\om^{\rho_0}\, \cJ_{k+1/2}^\la)(t) \wrt t.
$$
Note that
$P_{k+1}(\cO)^*(\om^{\rho_0}\, \cJ_{k+1/2}^\la)$ may be written as 
$$
\om^{\rho_0}\, P_{k+1}(\cO)^*(\cJ_{k+1/2}^\la)
+ \sum_{j=1}^k c_{j,k}'\, (\cO^j\om)^{\rho_0}
  \, (\cO^{k-j}\cJ_{k+1/2})^\la,
$$
for suitable constants $c_{j,k}'$, and that 
$P_{k+1}(\cO)^*(\cJ_{k+1/2}^\la)(t) 
= \cos(t\la)$, by Remark \ref{rem: P^*}. Hence
$$
\begin{aligned}
& S_0(\la;\vep) \\
&= \bigl[\wh\om_{\rho_0}*(\wh\om^\vep \,\wh K)\bigr](\la) 
  + \sum_{j=1}^k c_{j,k}\, \ir (\om_\vep *K)(t) \, 
 (\cO^j\om)^{\rho_0}(t) \,  (\cO^{k-j}\cJ_{k+1/2})^\la(t) \wrt t.
\end{aligned}
$$
Note that for each positive integer $j$ the function
$\cO^j\om$ vanishes in $[-1/4,1/4]$, and that the restriction
of $K$ to $[-1/4,1/4]^c$ is a bounded function by \rmi\ (with $k=0$).
Then it is straightforward to check that $S_0(\la;\vep)$ tends
to $S_0(\la)$ for all $\la$ in $\BR$.

To prove that $S_i(\la;\vep)$ tends
to $S_i(\la)$ for all $\la$ in $\BR$
and all $i$ in $\{1,\ldots,d\}$, observe that
$$
\begin{aligned}
2\pi\, S_i(\la;\vep)
& = \prodo{\phi^{\rho_i}\, 
      \cJ_{k+1/2}^\la}{P_{k+1}(\cO)(\om_\vep*K)} \\
& = \prodo{P_{k+1}(\cO)^*(\phi^{\rho_i}\, 
      \cJ_{k+1/2}^\la)}{\om_\vep*K},
\end{aligned}
$$
where $\prodo{\cdot}{\cdot}$ denotes the duality between
test functions and distributions on $\BR$.  Now we let $\vep$ tend 
to $0$ and obtain 
$$
\begin{aligned}
2\pi\, S_i(\la;\vep)
& \to  \prodo{P_{k+1}(\cO)^*(\phi^{\rho_i}\, 
      \cJ_{k+1/2}^\la)}{K} \\
& =  \prodo{\phi^{\rho_i}\, \cJ_{k+1/2}^\la}{P_{k+1}(\cO)K}.
\end{aligned}
$$
By \rmi\ the distribution $P_{k+1}(\cO)K$ is a bounded function
on the support of $\phi^{\rho_i}$, so that the right hand
side is exactly $2\pi\,S_i(\la)$, 
thereby concluding the proof of \rmii.

\medskip
Finally, to prove \rmiii, observe that
$$
\begin{aligned}
\mod{S_0(\la)}
& \leq \mod{(\wh\om_{\rho_0}*\wh K)(\la)}
  +  C \, \sum_{j=1}^k \ir \mod{K(t)} \, 
 \mod{(\cO^j\om)^{\rho_0}(t)}  \wrt t \\
& \leq C \, \norm{\wh K}{\infty} 
  +  C \, \norm{t K}{\infty} \,  \sum_{j=1}^k \ir \mod{t}^{-1}\,  
 \mod{(\cO^j\om)^{\rho_0}(t)}  \wrt t \\
& \leq C \, \norm{\wh K}{\Horm(2)} 
\quant \la \in \BR,
\end{aligned}
$$
as required.  We have used \rmi\ (with $k=0$) in the second inequality
above.
\end{proof}

\subsection{A remark on the wave propagator} \label{subsec: Mean}

We shall need to prove that certain operators map $H^1$-atoms
into $\hu{M}$.  In particular, we need to show that the image of
an atom $a$ has integral $0$.  

\medskip
\noindent
\textbf{Notation.}
For notational convenience, we denote by $\cD_1$
the operator $\sqrt{\cL-b+\kappa^2}$ ($\kappa$ is defined in the Basic
assumptions~\ref{Ba: on M}).
\medskip

Suppose that $\cT$ is an operator bounded on $\ld{M}$.  We
denote by $k_{\cT}$ its Schwartz kernel (with respect to the
Riemannian density $\mu$).  

\begin{proposition}  \label{p: Mean}
Suppose that $\nu$ is in $[-1/2,\infty)$, that $w$ is in $\lu{\BR}$,
and that $a$ is a $H^1$-atom.  Define the operator $\cW_\nu(\cD)$
on $\ld{M}$ spectrally by
$$
\cW_\nu(\cD) f 
= \ir w(t) \, \cJ_{\nu} (t\cD) f \wrt t
\quant f \in \ld{M}.
$$
The following hold:
\begin{enumerate}
\item[\itemno1]
$\int_M \cW_\nu(\cD) a \wrt \mu = 0$;
\item[\itemno2]
$\int_M S_0(\cD) a \wrt \mu = 0$ ($S_0$ is defined in (\ref{f: A0})).
\end{enumerate}
The same conclusions hold if we replace the operator $\cD$ by the
operator $\cD_1$. 
\end{proposition}

\begin{proof}
We observe preliminarly that if $a$ is a $H^1$-atom, then
\begin{equation} \label{f: mean value zero}
\int_{M} \cos(t\cD)  a \wrt \mu = 0
\quant t \in \BR^+.
\end{equation}
Indeed, $\cos(t\cD)  a$ is in $\ld{M}$, because $\cos(t\cD)$
is bounded on $\ld{M}$, and is supported in a ball of 
radius $t+r_B$, where $B$ is any ball that contains the
support of $a$.  Therefore, $\cos(t\cD)  a$ is in $\lu{M}$,
and 
$$
\int_{M} \cos(t\cD)  a \wrt \mu
= \lim_{N\to\infty} \int_M \One_{B(c_B,N)} \, \cos(t\cD)  a \wrt \mu.
$$
Now, the last integral is the inner product
$\bigl(\cos(t\cD)  a,  \One_{B(c_B,N)}\bigr)$ in $\ld{M}$, 
and is equal to $\bigl(a, \cos(t\cD) \One_{B(c_B,N)}  \bigr)$,
because $\cos(t\cD)$ is self adjoint.   Observe that 
$\cos(t\cD) \One_{B(c_B,N)}$ is equal to $\cosh(\sqrt{b}t)$
on $B(c_B,N-t)$, because both functions are solutions of  the wave equation 
$\partial^2_t u+\cL u=b u$ in $B(c_B,N)\times(0,\infty)$ and 
satisfy the same initial conditions $u(x,0)=1$, 
$\partial_t u(x,0)=0$ in $B(c_B,N)$. 
Hence, they coincide in $\set{(x,t): d(x,c_B)<N-t}$, by standard energy
estimates.    
If $N$ is so big that $B(c_B,N-t)$ contains the support of $a$, then
$$
\bigl(a,  \cos(t\cD) \One_{B(c_B,N)} \bigr)
= \cosh(\sqrt{b}t)\,\int_{M} a \wrt \mu = 0,
$$ 
and (\ref{f: mean value zero}) follows.

A straightforward consequence of (\ref{f: mean value zero}) is that 
for any $\nu$ in $(-1/2,\infty)$ and for every $H^1$-atom $a$
\begin{equation} \label{f: mean value zero I}
\int_{M} \cJ_{\nu}(t\cD) a \wrt \mu = 0
\quant t \in \BR^+.
\end{equation}
Indeed, 
$$
\cJ_{\nu} (t \cD) a
= \smallfrac{\nu+2}{\sqrt\pi \,\Ga(\nu+1/2)} 
  \int_0^1 (1-s^2)^{\nu-1/2} \, \cos(st\cD) a\wrt s,
$$
and the required conclusion follows from Fubini's Theorem.
It is straightforward to check that similar considerations 
apply to the operator $\cD_1$, so that for each $\nu$ in 
$[-1/2,\infty)$
$$
\int_{M} \cJ_{\nu}(t\cD_1) a \wrt \mu = 0
\quant t \in \BR^+.
$$

To prove \rmi\ we just observe that
$$
\begin{aligned}
\int_M \cW_{\nu} (\cD)a \wrt \mu
&= \int_M \wrt \mu \ir w(t) \, \cJ_\nu(t\cD)a \wrt t \\
&= \ir {\wrt t\, } w(t) \, \int_M \cJ_\nu(t\cD)a \wrt \mu
= 0,
\end{aligned}
$$
where the change of the order of integration is
justified by Fubini's theorem.

\medskip
Next we prove \rmii.  By (\ref{f: A0}), the function
$S_0(\cD)a$ may be written as the sum of  
$$
(\wh\om_{\rho_0}*\wh K)(\cD)a
\qquad\hbox{and}\qquad
\sum_{j=1}^k c_{j,k}\, \ir K(t) \, 
\cO^j\om(\rho_0t) \,  \cO^{k-j}\cJ_{k+1/2}(t\cD)a \wrt t,
$$
where $K$ is a compactly supported distribution on $\BR$ such that
$\wh K$ is bounded and $tK$ is in $\ly{\BR}$.
It is a straightforward consequence of \rmi\ that the integral of each 
summand of the sum above is equal to $0$.  Thus, to prove
that the integral of $S_0(\cD)a$ is $0$, it suffices to show that
the integral of $(\wh\om_{\rho_0}*\wh K)(\cD)a$ 
makes sense and is equal to $0$.  Since $\wh K$ is bounded,
$\om^\vep \, \wh K$ tends pointwise and boundedly to $\wh K$ as $\vep$
tends to $0$.  Then $\wh\om_{\rho_0}*(\om^\vep \, \wh K)$  
tends pointwise and boundedly to $\wh\om_{\rho_0}*\wh K$ as $\vep$
tends to $0$ by the Lebesgue dominated
convergence theorem.  Therefore the operator 
$\wh\om_{\rho_0}*(\om^\vep \, \wh K)(\cD)$ tends to
the operator $\wh\om_{\rho_0}*\wh K(\cD)$
in the strong operator topology of $\ld{M}$.
Consequently
$\wh\om_{\rho_0}*(\om^\vep \, \wh K)(\cD)a$ tends to
$\wh\om_{\rho_0}*\wh K(\cD)a$ in $\ld{M}$ as $\vep$ tends to $0$. 

Suppose that the support of $a$ is contained in the ball $B$.  Since 
the function $\om^{\rho_0} \, (\wh \om_{\vep}*K)$ is in $\lu{\BR}$,
$$
\bigl[\wh\om_{\rho_0}*(\om^\vep \wh K)\bigr](\cD)a
= \frac{1}{2\pi} \ir \om^{\rho_0}(t) \, (\wh \om_{\vep}*K)(t)  
\, \cos (t\cD)a \wrt t.
$$
Since the support of $\om^{\rho_0} \, (\wh \om_{\vep}*K)$ 
is contained in $[-1,1]$, all the functions
$\bigl[\wh\om_{\rho_0}*(\om^\vep \wh K)\bigr](\cD)a$ are supported
in the ball $B(c_B,r_B+1)$ by finite propagation speed, and 
$$
\int_M \bigl[\wh\om_{\rho_0}*(\om^\vep \wh K)\bigr](\cD)a \wrt \mu
= 0 
$$
by \rmi.  Thus, the function $\wh\om_{\rho_0}*\wh K(\cD)a$ is also 
supported in $B(c_B,r_B+1)$.  Hence   
$\wh\om_{\rho_0}*(\om^\vep \, \wh K)(\cD)a$ tends to
$\wh\om_{\rho_0}*\wh K(\cD)a$ in $\lu{M}$ as $\vep$ tends to $0$,
so that 
$$
\int_M (\wh\om_{\rho_0}*\wh K)(\cD)a  \wrt\mu
= \lim_{\vep \to 0} \int_M \wh\om_{\rho_0}*(\om^\vep \, \wh K)(\cD)a \wrt\mu 
= 0,
$$
as required to conclude the proof of \rmii. 
\end{proof}

\begin{remark}
Note that for every $\nu$ in $[-1/2,\infty)$
the function $\la\mapsto\cJ_\nu(t\la)$ is even and of
entire of exponential type $t$, so that 
kernel $k_{\cJ_\nu (t\cD)}$ of the operator 
$\cJ_\nu(t\cD)$ 
is supported in 
the set $\{(x,y) \in M\times M:
d(x,y) \leq t\}$ by the finite propagation speed.
A similar remark applies to the kernel of the
operator $\cJ_\nu(t\cD_1)$.
\end{remark}

\subsection{Economical decomposition of atoms} \label{subsec: Economical}

The following lemma produces an \emph{economical} decomposition
of atoms supported in ``big'' balls as finite linear combination
of atoms supported in balls of radius at most $1$, and is key to prove
Theorem~\ref{t: lim huno} below.  The idea is ``to transport charges 
along geodesics''.

\begin{lemma} \label{l: economical decomposition}
There exists a constant $C$ such that for every $H^1$-atom $a$ 
supported in a  ball $B$ of radius $r_B>1$
$$\norm{a}{H^1}\le C\,r_B,$$
where $\norm{a}{H^1}$ is the atomic norm in $\hu{M}$ associated
to the scale $1$.
\end{lemma}

\begin{proof}
Denote by $\fS$ a $1/3$-discretisation 
of $M$, i.e. a set of points  in $M$ that is
maximal with respect to the property
$$
\min\{d(z,w): z,w \in \fS, z \neq w \} >1/3,
\quad\hbox{and}\quad d(\fS, x) \leq 1/3 \quant x \in M.  
$$
The family $\set{B(z,1):z\in \fS}$ is a covering of $M$ which is 
uniformly locally finite,  by the uniform ball size and the locally 
doubling properties. By the same token, the set $B\cap\fS$ is  
finite and has at most $N$ points $z_1,\ldots,z_N$, with $N \le 
C\,\mu(B)$, where $C$ is a constant which does not depend on $B$. 
Denote  by $B_j$  the ball with centre $z_j$ and radius $1$,
and by $\set{\psi_j: j=1,\ldots,N}$ a partition of unity on $B$ 
subordinated to the covering $\set{B_j:j=1,\ldots,N}$.

Fix~$j$ in $\{1,\ldots,N\}$ and denote by $z_j^0,\ldots, z_j^{N_j}$
points on a minimizing geodesic joining $z_j$ and $c_B$, 
with the property that 
$z_j^0 = z_j$, $z_j^{N_j} = c_B$, and $d(z_j^h,z_j^{h+1})$ is 
approximately equal to $1/3$. Note that $N_j\le 4r_B$.  
Denote by $B^h_j$ the ball $B(z^h_j,1/12)$, 
for $j=1,\ldots,N$ and $h=0,\ldots, N_j$. Then the balls 
$B^h_j$ are disjoint, $B^h_j\subset B(z_j^h,1)\cap B(z_j^{h+1},1)$ 
and $B_j^{N_j}=B(c_B,1/12)$. \par
Denote by $\phi^h_j$ a nonnegative function in $C^\infty_c(B^h_j)$ that
has integral $1$. By the uniform ball size property we may choose the 
functions $\phi^h_j$ so that there exists a constant $A$ 
such that $\norm{\phi^h_j}{2}\le A$ for all $h$ and $j$.\par
Now, denote by $a_j^0$ the function $a\, \psi_j$.  Clearly
$$
a 
= \sum_{j=1}^N \psi_j \, a
= \sum_{j=1}^N  \, a_j^0. 
$$
Next, define
$$
a_j^1
= a_j^0-\phi_j^0 \, \int_M a_j^0\wrt\mu
\quad\hbox{and}\quad
a_j^h=(\phi_j^{h-2}-\phi_j^{h-1})\int_M a_j^0\wrt\mu,
\quad 2\le h\le N_j+1.
$$
Then, for every $h$ in $\set{1,\ldots,N_j}$,
the support of $a_j^h$ is contained in $B(z_j^{h-1},1)$, the integral 
of  $a_j^h$ vanishes and  
$$
\begin{aligned}
\norm{a_j^h}{2}
& \le 2A \int_M \mod{a_j^0}\wrt\mu \\
& \le C\, \norm{a_j^0}{2} \, \mu(B_j)^{1/2}\\
& \le C\, \norm{a_j^0}{2} \, \mu(B_j^h)^{-1/2}.
\end{aligned}
$$ 
In the last two inequalities we have used the fact that for each $r$
in $\BR^+$ the supremum of $\mu(B)$ over all balls $B$ of radius~$r$
is finite by the uniform ball size property.   
Hence there exists a constant $C$, independent of $j$ and $h$,
such that 
\begin{equation} \label{f: normaHuno ajk}
\norm{a_j^h}{H^1}
\leq C \, \norm{a_j^0}{2}.
\end{equation}
Moreover
$$
a_j^0=\sum_{h=1}^{N_j+1}a_j^h+\phi_j^{N_j}\int_M a_j^0\wrt\mu.
$$
Thus
$$
a
= \sum_{j=1}^N\sum_{h=1}^{N_j+1} a^h_j,
$$
because $\sum_j \int_M a^0_j\wrt\mu=\int_M a\wrt\mu=0$ and all the functions 
$\phi_j^{N_j}$, $j=1,\ldots,N_j$ coincide, for $B^{N_j}_j = B(c_B,1/12)$.
Now we use (\ref{f: normaHuno ajk}) and the fact that $N_j\le C\,r_B$, 
and conclude that
$$
\begin{aligned}
\norm{a}{H^1} 
& \leq C \,   \sum_{j=1}^N \sum_{h=1}^{N_j+1} \norm{a_j^0}{2}\\
& \leq    C \, r_B \,  \sum_{j=1}^N \norm{a_j^0}{2}.
\end{aligned}
$$
Then we use Schwarz's inequality and the fact that 
$N\leq C \, \mu(B)$, and obtain that 
$$
\begin{aligned}
\norm{a}{H^1} 
& \leq C \,r_B\, N^{1/2} \,   
    \Bigl(\sum_{j=1}^N \norm{a_j^0}{2}^2 \Bigr)^{1/2}  \\
& \leq C \, r_B\, \mu(B)^{1/2} \,  \norm{a}{2}  \\
& \leq C \, r_B.
\end{aligned}
$$
The last inequality follows because $a$ is a $H^1$-atom 
supported in the ball $B$.   

This completes the proof of the lemma.
\end{proof}

\subsection{Proof of Theorem~\ref{t: lim huno}}  \label{subsec: Hone}

For the reader's convenience, we recall 
one of the properties of functions in $H^\infty(\bS_{W};J)$
(see Definition~\ref{d: Hormander at infinity}), 
which will be key in the proof of 
Theorem~\ref{t: lim huno}.

\begin{lemma}[{\cite[Lemma~5.4]{HMM}}] \label{l:splitf}
Suppose that $J$ is an integer $\geq 2$, and that $W$ is in $\BR^+$.
Then there exists a positive constant $C$ such that 
for every function $f$ in $H^\infty\bigl(\bS_{W};J\bigr)$, and
for every positive integer $h \leq J-2$
$$
\mod{\cO^h\widehat f (t)} 
\leq C\, \norm{f}{\bS_{W};J}\,  \mod{t}^{h-J} \, \e^{-W \mod{t}}
\quant t \in \BR\setminus\{0\}.
$$
\end{lemma}

\noindent
We restate Theorem~\ref{t: lim huno} for the reader's
convenience. 

\begin{theorem*} {\bf{3.5}}
Assume that $\al$ and $\be$ are as in (\ref{f: volume growth}),
and $\de$ as in (\ref{f: special}). 
Denote by $N$ the integer $[\!\![n/2+1]\!\!] +1$.  
Suppose that $J$ is an integer $>\max\, \bigl(N+2+\al/2-\de,
N + 1/2\bigr)$.  Then there exists a constant $C$ such that 
$$
\opnorm{m(\cD)}{H^1}
\leq C \, \norm{m}{\bS_{\be};J}
\quant m\in H^\infty\bigl(\bS_{\be};J\bigr).
$$
\end{theorem*}

\begin{proof}
For notational convenience, in this proof we shall write $\cJ$
instead of $\cJ_{N-1/2}$.

\medskip
\emph{Step I: reduction of the problem}.
We claim that it suffices to prove that for each $H^1$-atom $a$ the function 
$m(\cD)\, a$ may be written
as the sum of atoms with supports contained in balls of $\cB_1$, 
with $\ell^1$ norm of the coefficients controlled
by $C \, \norm{m}{\bS_{\be};J}$.

Indeed, by arguing as in \cite[Thm~4.1]{MSV}, we may then
show that $m(\cD)$ extends to a bounded operator from $\hu{M}$ to $\lu{M}$,
with norm dominated by $C\, \norm{m}{\bS_{\be};J}$.
Note that \cite[Thm~4.1]{MSV} is stated for spaces of
homogeneous type.  However, its proof extends to the present setting.  
Now, suppose that $f$ is a function in $\hu{M}$ and that
$f = \sum_j \la_j \, a_j$ is an atomic decomposition of $f$ with
$\norm{f}{H^1}\geq \sum_j \mod{\la_j} -\vep$.  Then 
$m(\cD) f = \sum_j \la_j \, m(\cD) a_j$, where the 
series is convergent in $\lu{M}$, because $m(\cD)$ extends to a
bounded operator from $\hu{M}$ to $\lu{M}$.  But the partial sums of
the series $\sum_j \la_j \, m(\cD) a_j$ is a Cauchy sequence in $\hu{M}$,
hence the series is convergent in $\hu{M}$, and the sum must be
the function $m(\cD)f$.
Then 
$$
\begin{aligned}
\norm{m(\cD)f}{H^1}
& \leq \sum_j \mod{\la_j} \, \norm{m(\cD)a_j}{H^1} \\
& \leq C\, \, \norm{m}{\bS_{\be};J}\,  \sum_j \mod{\la_j} \\
& \leq C\, \, \norm{m}{\bS_{\be};J}\, (\norm{f}{H^1}+ \vep),
\end{aligned}
$$
and the required conclusion follows by taking
the infimum of both sides with respect to all admissible
decompositions of $f$.  

\medskip
\emph{Step II: splitting of the operator}.
Let $\om$ be the cut-off function defined in Section 3.
Clearly $\wh\om\ast m$ and $m-\wh\om\ast m$ are bounded functions.
Define the operators~$\cS $ and $\cT $ spectrally by 
$$
\cS 
= (\wh\om\ast m) (\cD)
\qquad\hbox{and}\qquad
\cT 
= (m-\wh\om\ast m) (\cD).
$$
Then
$
m(\cD) 
= \cS  + \cT .
$
We analyse the operators $\cS $ and $\cT $ in Step~III and Step~IV
respectively. 

\medskip
Suppose that $a$ is a $H^1$-atom
supported in $B(p,R)$ for some $p$ in~$M$ and~$R\leq 1$.  
\medskip

\medskip
\emph{Step III: analysis of $\cS $}.
In the following, we shall need to estimate the $\ld{M}$ norm
of the differential of the kernel of 
certain operators related to $\cS $.  To this end, and to be able 
to apply \cite[Proposition~2.2~\rmiii]{MMV1}, we write the operator
$\cS $ as a function of the operator $\cD_1$, rather than of $\cD$.
Recall that $\cD_1=\sqrt{\cD^2+\kappa^2}$.

Since $\wh{\om}\ast m$
is an \emph{even} entire function of exponential type $1$,
the function
$S $, defined by 
$$
S  (\zeta)
= (\wh{\om}\ast m)\bigl(\sqrt{\zeta^2-\kappa^2}\bigr) 
\quant \zeta \in \BC,
$$
is well defined, and is of exponential type $1$. 
Hence its Fourier transform has support in $[-1,1]$. 
It is straightforward to check that 
$$
\cS  = S  (\cD_1),
$$
and that
$$
\norm{S }{\Horm(J)}
\leq C\, \norm{\wh{\om}\ast m}{\Horm(J)},
$$
where the constant $C$ does not depend on $m$.  
By arguing much as in the proof of \cite[Proposition~5.3]{HMM},
we may show that 
$\norm{\wh{\om}\ast m}{\Horm(J)} \leq C\, \norm{m}{\Horm(J)}$,
where $C$ is independent of $m$.  Clearly
$$
\norm{m}{\Horm(J)}
\leq \norm{m}{\bS_{\be};J}
\quant m \in H^\infty(\bS_{\be};J).
$$
Hence there exists a constant $C$ such that 
\begin{equation} \label{f: rel Horm}
\norm{S }{\Horm(J)}
\leq C \,  \norm{m}{\bS_{\be};J}
\quant m \in H^\infty(\bS_{\be};J).
\end{equation}

Define the functions $S_i$ as in (\ref{f: A0}) and
(\ref{f: Ai}), but with $N-1$ in place of $k$ and 
the Fourier transform of $S$ in place of $K$.
We further decompose $\cS $ as $\sum_{i=0}^{d} S_i(\cD_1)$,
where $d$ is as in (\ref{f: dec om phi}).
The function $S_0$ is bounded by Lemma~\ref{l: technical nuovo}~\rmiii, 
hence $S_0(\cD_1)$ is bounded on $\ld{M}$ by the spectral theorem, and
$$
\opnorm{S_0(\cD_1)}{2}\
\leq \norm{S_0}{\infty}
\leq C \, \norm{S}{\Horm(2)}
\leq C \, \norm{m}{\bS_{\be};J}.  
$$
Observe that the support of the kernel of the operator $S_{i}(\cD_1)$ 
is contained in $\{(x,y): d(x,y) \leq 4^{i+1}R\}$ by the finite
propagation speed.  
Thus the support of $S_{i}(\cD_1)a$ is contained in the 
ball with centre~$p$ and radius $(4^{i+1}+1)R$, 
which henceforth we denote by $B_i$. 
In particular $S_{0}(\cD_1)a$ is supported in $B_0=B(p,5R)$, and
$$
\norm{S_{0}(\cD_1)a}{2}
\leq C\, \opnorm{S_{0}(\cD_1)}{2}\, \norm{a}{2} 
\leq C\, R^{-n/2} \, \norm{m}{\bS_{\be};J}.
$$
Furthermore, the integral of $S_0(\cD_1)a$ vanishes
by Proposition~\ref{p: Mean}~\rmii, 
so that $S_0(\cD_1)\,a$ is a constant multiple of a $H^1$-atom.

Denote by $k_{S_i(\cD_1)}$ the integral kernel of the
operator $S_i(\cD_1)$.  Observe that
$$
S_i(\cD_1)\, a(x)
= \int_{B(p,R)} a(y) \, \bigl[k_{S_i(\cD_1)}(x,y)
   -  k_{S_i(\cD_1)}(x,p)\bigr] \wrt \mu(y).
$$
By Minkowski's integral inequality
and the fact that the support of $S_i(\cD_1)\, a$
is contained in $B_i$, we have that
$$
\begin{aligned}
\norm{S_i(\cD_1)\, a}{2} 
& = \norm{S_i(\cD_1)\, a}{\ld{B_i}}  \\
& \leq \int_{B(p,R)} \mod{a(y)} \, I_{i}(y) \wrt \mu(y),
\end{aligned}
$$
where
$$
I_{i}(y) 
= \norm{k_{S_i(\cD_1)}(\cdot,y) 
- k_{S_i(\cD_1)}(\cdot,p)}{L^2(B_i)}
\quant y \in B(p,R).
$$
To estimate $I_{i}(y)$, we observe that 
$$
I_{i}(y)
\leq  d(y,p) \, \sup_{z\in M} \, \bignorm{\dest_2 
    k_{S_i(\cD_1)}(\cdot,z)}{\ld{B_{i}}}
$$
and, by Lemma~\ref{l: technical nuovo}~\rmii\ (with $k=N-1$),
$$
\dest_2 
    k_{S_i(\cD_1)}(\cdot,z)=\frac{1}{2\pi} \, \ir \phi^{\rho_i}(t)\, 
   P_{N}(\cO)\widehat{S}(t) \  \dest_2 
    k_{\cJ(t\cD_1)}(\cdot,z) \wrt t.
$$
Recall that $\phi^{\rho_i}$ is supported in $E_i
= \{ t\in \BR: 4^{i-1} R \leq \mod{t} \leq 4^{i+1}R \}$, 
that the support of $\widehat{S}$ is contained in $[-1,1]$ and that $d(p,y)<R$.
Then, by \cite[Proposition~2.2~\rmii]{MMV1} (with $\cJ$
in place of $F$), there exists a constant $C$, independent of $i$
and $R$, such that
$$
\begin{aligned}
I_{i}(y)
& \leq C \, d(y,p) \, 
    \ir \phi^{\rho_i}(t)\, \mod{P_{N}\cO)\wh S(t)} \, \,
    \sup_{z\in M} \bignorm{\dest_2
     k_{\cJ (t\cD_1)}(\cdot,z)}{\ld{B_{i}}} \wrt t \\
& \leq C \, \norm{t P_{N}(\cO)\wh S}{\infty} \,R \,    
     \int_{E_{i}} \mod{t}^{-n/2-2}   \wrt t \\
& \leq C \, \norm{m}{\bS_{\be};J}\, R \, (4^iR)^{-n/2-1} \,.
\end{aligned}
$$
Thus,
$$
\begin{aligned}
\norm{S_i(\cD_1)\, a}{2}
& \leq C\, \, \norm{m}{\bS_{\be};J} \, 4^{-i}\, 
      (4^i R)^{-n/2} \, \norm{a}{1} \\
& \leq C\, \, \norm{m}{\bS_{\be};J} \, 
      4^{-i}\, \mu(B_i)^{-1/2}.
\end{aligned}
$$
Furthermore the integral of $S_i(\cD_1)\, a$ vanishes
by Proposition~\ref{p: Mean}~\rmi, so that 
the function $4^i \, S_i(\cD_1)\, a$ is a constant multiple of a $H^1$-atom.
Thus
$$
\begin{aligned}
\norm{\cS \, a}{H^1}
& \leq C \, \, \norm{m}{\bS_{\be};J}\sum_{i=0}^\infty 4^{-i} \\
& \leq C \,  \norm{m}{\bS_{\be};J}.
\end{aligned}
$$

\medskip
\emph{Step IV: analysis of $\cT$}.
For each $j$ in $\{1,2,3,\ldots\}$, define $\om_j$ by the formula
\begin{equation}\label{omj}
\om_j(t) = \om (t-j) + \om(t+j) \quant t \in \BR.
\end{equation}
Observe that $\sum_{j=1}^\infty \om_j=1-\om$
and that the support of $\om_j$ is contained in the set of all 
$t$ in $\BR$ such that $j-3/4\le\mod{t}\le j+3/4$.

Since $m$ is in $H^\infty\bigl(\bS_{\be};J\bigr)$ and $J \geq N+2$, 
the function $\wh m$ and its derivatives up to the order $N$ are 
rapidly decreasing at infinity 
by Lemma~\ref{l:splitf}, so that 
$\cO^\ell(\om_j\, \wh m)$ is in $\lu{\BR}\cap C_0(\BR^+)$ 
for all $\ell$ in $\{0,\ldots, N\}$, and so does 
$P_N(\cO)(\om_j\, \wh m)$.
In the rest of this proof, we write $\Om_{j,N}$ instead of 
$P_N(\cO)(\om_j\, \wh m)$.  
Observe that the support of $\Om_{j,N}$ 
is contained in $\set{t\in\BR:j-3/4\le\mod{t}\le j+3/4}$. 

Define the function $T_j: \BR\to \BC$ by
\begin{equation} \label{f: Bj}
T_j(\la) 
= \ir  \Om_{j,N}(t)\, \cJ  (t \la)\wrt t \quant \la \in \BR.
\end{equation}
We may use the observation that $(m-\wh\om\ast m)\wh{\phantom a} 
= \sum_{j=1}^\infty \om_j\, \wh m$ 
and formula~(\ref{f: propertiesRiesz I}), and write
$$
\begin{aligned}
(m-\wh\om\ast m)(\la)
& = \frac{1}{2\pi} \ir  \bigl(1-\om(t)\bigr) \, \wh{m}(t)
      \, \cos (t\la) \wrt t \\ 
& = \sum_{j=1}^\infty T_j(\la). 
\end{aligned}
$$
Then, by the spectral theorem,
$$
\cT  a
=\sum_{j=1}^\infty T_j(\cD) a.
$$
By the asymptotics of $J_{N-1/2}$ \cite[formula (5.11.6), p.~122]{L} 
$$
\sup_{s>0} \mod{(1+s)^{N} \, \cJ  (s)} < \infty.
$$
Since $N-1/2>(n+1)/2$, we may apply 
\cite[Proposition~2.2~\rmi]{MMV1} and conclude that 
$$
\begin{aligned}
\norm{\cJ (t\cD) a}{2}
& \leq \norm{a}{1} \, \bigopnorm{\cJ (t\cD)}{1;2} \\
& \leq \sup_{y\in M}  \bignorm{k_{\cJ (t\cD)}(\cdot,y)}{2} \\
& \leq C\, \mod{t}^{-n/2}\, \bigl(1+\mod{t}\bigr)^{n/2-\de} 
\quant t \in \BR\setminus \{0\}.
\end{aligned}
$$
Then $\cJ (t\cD)a$ is supported in $B(p,t+R)$,
and has integral $0$ by Proposition~\ref{p: Mean}~\rmi.
Observe that 
\begin{align}
\norm{T_j(\cD)a}{2}
& \leq C\,   \ir  \mod{\Om_{j,N}(t)}  
     \,  \norm{\cJ  (t \cD)a}{2}  \wrt t \nonumber \\
& \leq C\, \int_{j-3/4}^{j+3/4} \mod{{\Om_{j,N}(t)}}  \,
         \mod{t}^{-n/2}\, \bigl(1+\mod{t}\bigr)^{n/2-\de} \wrt t \\
& \leq  C \, \norm{m}{\bS_{\be};J} \, \, j^{N-J-\de}\, \e^{-\be \, j}
    \quant j \in \{1,2,\ldots\}\nonumber.
\end{align}
In the last inequality we have used 
Lemma~\ref{l:splitf} and \cite[Proposition~2.2~\rmi]{MMV1}.
Note that 
$
j^{\de+J-N-\al/2}\,  T_j(\cD)a
$
is a constant multiple of a $H^1$-atom.
Indeed,~$T_j(\cD)a$ is a function in $\ld{M}$
with support contained in $B\bigl(p, j+1\bigr)$, and has integral~$0$
by Proposition~\ref{p: Mean}~\rmi.  
Moreover
$$
\begin{aligned}
\norm{j^{\de+J-N-\al/2}\,  T_j(\cD)a}{2}
& \leq C\, \norm{m}{\bS_{\be};J} \,\,  j^{-\al/2} \, \e^{-\be\, j}  \\
& \leq  {C}\, \norm{m}{\bS_{\be};J} \,\,  {\mu\bigl(B(p,j+1)\bigr)^{-1/2}}
    \quant j \in \{1,2,\ldots\}.
\end{aligned}
$$
Hence we may write 
$$
\cT  a  
=  \sum_{j=1}^\infty \, \la_j \, a_j',
$$
where $a_j'$ is a $H^1$-atom supported in $B\bigl(p, j+1\bigr)$, and 
$$\la_j = C\, \norm{m}{\bS_{\be};J} \,j^{N+\al/2-J-\de}.$$
By Lemma~\ref{l: economical decomposition} we have
$
\norm{a_j'}{H^1}
\leq C \, j,
$
so that 
$$
\begin{aligned}
\norm{\cT  a}{H^1}  
& \leq  \sum_{j=1}^\infty \, \mod{\la_j} \, \norm{a_j'}{H^1} \\
& \leq  C\, \norm{m}{\bS_{\be};J} \,\,  
      \sum_{j=1}^\infty \, j^{1+N+\al/2-J-\de}, 
\end{aligned}
$$
which is finite (and independent of $a$) because $J>2+N+\al/2-\de$.

\medskip
\emph{Step V: conclusion}.  By Step~III and Step~IV there exists
a constant $C$ such that for every $H^1$-atom $a$ with 
support contained in a ball of radius at most $1$
$$
\norm{\cS a}{H^1} + \norm{\cT a}{H^1}
\leq  C\, \norm{m}{\bS_{\be};J}. 
$$ 
Then Step~II implies that
$$
\norm{m(\cD)a}{H^1}
\leq  C\, \norm{m}{\bS_{\be};J}. 
$$ 
The required conclusion follows from Step~I. 
\end{proof}

\end{document}